\documentclass{amsart}
\usepackage{amstext,amssymb,amsthm,amsopn,newlfont,graphpap,graphics,graphicx,mathrsfs,enumitem}
\allowdisplaybreaks
\usepackage[parfill]{parskip}
\usepackage[noadjust]{cite}
\usepackage{epigraph}
\usepackage[colorlinks=true,
            linkcolor=red,
            urlcolor=blue,
            citecolor=magenta]{hyperref}
\usepackage{color}
\usepackage{mathrsfs}
\allowdisplaybreaks
\theoremstyle{plain}
\newtheorem{thm}{Theorem}[section]
\newtheorem*{thm*}{Theorem}
\newtheorem{prop}{Proposition}[section]
\newtheorem*{prop*}{Proposition}
\newtheorem{cor}{Corollary}[section]
\newtheorem*{cor*}{Corollary}

\newtheorem*{lem*}{Lemma}
\theoremstyle{definition}
\newtheorem{defn}{Definition}[section]
\newtheorem*{defn*}{Definition}

\newtheorem*{exmps*}{Examples}

\newtheorem*{exmp*}{Example}

\newtheorem*{exerc*}{Exercise}

\newtheorem{rems}{Remarks}[section]
\newtheorem*{rems*}{Remarks}
\newtheorem{rem}[rems]{Remark}
\newtheorem*{rem*}{Remark}
\newcommand{\N}{{\mathbb N}}
\newcommand{\Z}{{\mathbb Z}}
\newcommand{\R}{{\mathbb R}}
\newcommand{\C}{{\mathbb C}}

\newcommand\restr[2]{\ensuremath{#1\big|{#2}}}

\DeclareMathOperator{\Rep}{Re\ignorespaces}
\DeclareMathOperator{\Imp}{Im\ignorespaces}
\DeclareMathOperator{\dist}{dist}
\DeclareMathOperator{\spa}{span}
\begin{document}
\title[On the Gevrey ultradifferentiability of weak solutions]
{On the Gevrey ultradifferentiability\\ 
of weak solutions\\
of an abstract evolution equation\\
with a scalar type spectral operator\\
on the real axis}
\author[Marat V. Markin]{Marat V. Markin}
\address{
Department of Mathematics\newline
California State University, Fresno\newline
5245 N. Backer Avenue, M/S PB 108\newline
Fresno, CA 93740-8001, USA
}
\email{mmarkin@csufresno.edu}
\subjclass{Primary 34G10, 47B40, 30D60, 30D15; Secondary 47B15, 47D06, 47D60}
\keywords{Weak solution, scalar type spectral operator, Gevrey classes}
\begin{abstract}
Given the abstract evolution equation
\[
y'(t)=Ay(t),\ t\in \mathbb{R},
\]
with a \textit{scalar type spectral operator} $A$ in a complex Banach space, we find conditions on $A$, formulated exclusively in terms of the location of its spectrum in the complex plane, \textit{necessary and sufficient} for all \textit{weak solutions} of the equation, which a priori need not be strongly differentiable, to be strongly Gevrey ultradifferentiable of order $\beta\ge 1$, in particular \textit{analytic} or \textit{entire}, on $\mathbb{R}$. We also reveal certain inherent smoothness improvement effects and show that, if all weak solutions of the equation are Gevrey ultradifferentiable of orders less than one, then the operator $A$ is necessarily \textit{bounded}. The important particular case of the equation with a \textit{normal operator} $A$ in a complex Hilbert space follows immediately.
\end{abstract}
\maketitle
\section[Introduction]{Introduction}

We find conditions on a scalar type spectral operator $A$ in a complex Banach space, formulated exclusively in terms of the location of its {\it spectrum} in the complex plane, \textit{necessary and sufficient} for all \textit{weak solutions} of the evolution equation
\begin{equation}\label{1}
y'(t)=Ay(t),\ t\in \R,
\end{equation}
which a priori need not be strongly differentiable, to be strongly Gevrey ultradifferentiable of order $\beta\ge 1$, in particular \textit{analytic} or \textit{entire}, on $\mathbb{R}$. We also reveal certain inherent smoothness improvement effects and show that, if all weak solutions of the equation are Gevrey ultradifferentiable of orders less than one, then the operator $A$ is necessarily \textit{bounded}. 

The important particular case of the equation with a \textit{normal operator} $A$ in a complex Hilbert space follows immediately.

We proceed along the path of developing the results of paper \cite{Markin2018(3)} on the strong differentiability of the weak solutions of equation \eqref{1} on $\R$ and of papers \cite{Markin2018(4),Markin2019(2),Markin2019(1)}, where similar consideration is given to the Gevrey ultradifferentiability of the weak solutions 
of the evolution equation
\begin{equation}\label{+}
y'(t)=Ay(t),\ t\ge 0,
\end{equation}
on $[0,\infty)$ and $(0,\infty)$.

\section[Preliminaries]{Preliminaries}

Here, we briefly outline certain facts essential for the subsequent discourse (for more, see, e.g., \cite{Markin2018(2),Markin2018(3),Markin2018(4),Markin2019(1)}).

\subsection{Weak Solutions}\ 

\begin{defn}[Weak Solution]\label{ws}\ \\
Let $A$ be a densely defined closed linear operator in a Banach space $(X,\|\cdot\|)$ and $I$ be an interval of the real axis $\R$. A strongly continuous vector function $y:I\rightarrow X$ is called a {\it weak solution} of
the evolution equation 
\begin{equation}\label{2}
y'(t)=Ay(t),\ t\in I,
\end{equation}
if, for any $g^* \in D(A^*)$,
\begin{equation*}
\dfrac{d}{dt}\langle y(t),g^*\rangle = \langle y(t),A^*g^* \rangle,\ t\in I,
\end{equation*}
where $D(\cdot)$ is the \textit{domain} of an operator, $A^*$ is the operator {\it adjoint} to $A$, and $\langle\cdot,\cdot\rangle$ is the {\it pairing} between
the space $X$ and its dual $X^*$ (cf. \cite{Ball}).
\end{defn}

\begin{rems}\label{remsws}\
\begin{itemize}
\item Due to the \textit{closedness} of $A$, a weak solution of equation \eqref{2} can be equivalently defined to be a strongly continuous vector function $y:I\mapsto X$ such that,
for all $t\in I$,
\begin{equation*}
\int_{t_0}^ty(s)\,ds\in D(A)\quad \text{and} \quad y(t)=y(t_0)+A\int_{t_0}^ty(s)\,ds,
\end{equation*}
where $t_0$ is an arbitrary fixed point of the interval $I$, and is also called a \textit{mild solution} (cf. {\cite[Ch. II, Definition 6.3]{Engel-Nagel}}, see also {\cite[Preliminaries]{Markin2018(2)}}).
\item Such a notion of \textit{weak solution}, which need not be differentiable in the strong sense, generalizes that of \textit{classical} one, strongly differentiable on $I$ and satisfying the equation in the traditional plug-in sense, the classical solutions being precisely the weak ones strongly differentiable on $I$.
\item As is easily seen $y:\R\to X$ is a weak solution of equation \eqref{1} \textit{iff} 
\begin{enumerate}[label=(\roman*)]
\item 
\[
y_+(t):=y(t),\ t\ge 0,
\]
is a weak solution of equation \eqref{+} and 
\[
y_-(t):=y(-t),\ t\ge 0,
\]
is a weak solution  of the equation
\begin{equation}\label{-}
y'(t)=-Ay(t),\ t\ge 0,
\end{equation}
or
\item \[
y_-(t):=y(-t),\ t\in \R,
\]
is a weak solution  of the equation
\begin{equation*}
y'(t)=-Ay(t),\ t\in \R,
\end{equation*}
\end{enumerate} 
\item When a closed densely defined linear operator $A$
in a complex Banach space $X$ generates a strongly continuous group $\left\{T(t) \right\}_{t\in \R}$ of  bounded linear operators (see, e.g., \cite{Hille-Phillips,Engel-Nagel}), i.e., the associated \textit{abstract Cauchy problem} (\textit{ACP})
\begin{equation}\label{ACP}
\begin{cases}
y'(t)=Ay(t),\ t\in \R,\\
y(0)=f
\end{cases}
\end{equation}
is \textit{well posed} (cf. {\cite[Ch. II, Definition 6.8]{Engel-Nagel}}), the weak solutions of equation \eqref{1} are the orbits
\begin{equation}\label{group}
y(t)=T(t)f,\ t\in \R,
\end{equation}
with $f\in X$ (cf. {\cite[Ch. II, Proposition 6.4]{Engel-Nagel}}, see also {\cite[Theorem]{Ball}}), whereas the classical ones are those with $f\in D(A)$
(see, e.g., {\cite[Ch. II, Proposition 6.3]{Engel-Nagel}}). 
\end{itemize} 
\end{rems} 

\subsection{Scalar Type Spectral Operators}\ 

Henceforth, unless specified otherwise, $A$ is a {\it scalar type spectral operator} in a complex Banach space $(X,\|\cdot\|)$ with strongly $\sigma$-additive \textit{spectral measure} (the \textit{resolution of the identity}) $E_A(\cdot)$ assigning to Borel sets of the complex plane $\C$ bounded projection operators on $X$ and having the operator's \textit{spectrum} $\sigma(A)$ as its {\it support} \cite{Dunford1954,Survey58,Dun-SchIII}.

Observe that, in a complex Hilbert space, the scalar type spectral operators are precisely all those that are similar to the {\it normal} ones \cite{Dun-SchII,Plesner,Wermer}.

Associated with a scalar type spectral operator $A$ is the {\it Borel operational calculus} assigning to each complex-valued Borel measurable function
$F:\sigma(A)\to \C$ a scalar type spectral operator
\[
F(A):=\int\limits_{\sigma(A)} F(\lambda)\,dE_A(\lambda)
\]
\cite{Dun-SchIII}. In particular,
\begin{equation}\label{exp}
A^n=\int\limits_{\sigma(A)} \lambda^n\,dE_A(\lambda),\ n\in\Z_+,\quad\text{and}\quad
e^{tA}:=\int\limits_{\sigma(A)} e^{t\lambda}\,dE_A(\lambda),\ t\in \R,
\end{equation}
($\Z_+:=\left\{0,1,2,\dots\right\}$ is the set of nonnegative integers, $A^0:=I$, $I$ is the \textit{identity operator} on $X$).

Provided
\[
\sigma(A)\subseteq \left\{\lambda\in\C\,\middle|\, \Rep\lambda\le \omega\right\}
\] 
with some $\omega\in \R$, the collection 
of exponentials $\left\{e^{tA}\right\}_{t\ge 0}$
is the $C_0$-\textit{se\-migroup} generated by $A$ {\cite[Proposition $3.1$]{Markin2002(2)}} (cf. also \cite{Berkson1966,Panchapagesan1969}), and hence, if
\[
\sigma(A)\subseteq \left\{\lambda\in\C\,\middle|\, -\omega\le \Rep\lambda\le \omega\right\}
\] 
with some $\omega\ge 0$, the collection of exponentials
$\left\{e^{tA}\right\}_{t\in\R}$ 
is the \textit{strongly continuous group} of bounded linear operators generated by $A$.

Being strongly $\sigma$-additive, the spectral measure is bounded, i.e., there exists an $M\ge 1$ such that, for any Borel set $\delta\subseteq \C$,
\begin{equation}\label{bounded}
\|E_A(\delta)\|\le M
\end{equation}
\cite{Dun-SchI,Dun-SchIII}.

\begin{rem}
The notation $\|\cdot\|$ is used here to designate the norm on the space $L(X)$ of all bounded linear operators on $X$. Henceforth, we adhere to this rather conventional economy of symbols adopting the same notation also for the norm on the dual space $X^*$.
\end{rem} 

For arbitrary Borel measurable function $F:\C\to \C$, $f\in D(F(A))$, $g^*\in X^*$, and Borel set $\delta\subseteq \C$,
\begin{equation}\label{cond(ii)}
\int\limits_\delta|F(\lambda)|\,dv(f,g^*,\lambda)
\le 4M\|E_A(\delta)F(A)f\|\|g^*\|,
\end{equation}
where $v(f,g^*,\cdot)$ is the \textit{total variation measure} of the complex-valued Borel measure $\langle E_A(\cdot)f,g^* \rangle$, for which
\begin{equation}\label{tv}
v(f,g^*,\C)=v(f,g^*,\sigma(A))\le 4M\|f\|\|g^*\|,
\end{equation}
where $M\ge 1$ in \eqref{cond(ii)} and \eqref{tv}
is from \eqref{bounded} (see, e.g., \cite{Markin2004(1),Markin2004(2)}). 

In particular, for $\delta=\sigma(A)$, $E_A(\delta)=I$ (see, e.g., \cite{Dun-SchIII}), \eqref{cond(ii)} turns into
\begin{equation}\label{cond(i)}
\int\limits_{\sigma(A)}|F(\lambda)|\,d v(f,g^*,\lambda)\le 4M\|F(A)f\|\|g^*\|.
\end{equation}

Further (see, e.g., \cite{Markin2018(3),Markin2018(4)}), for arbitrary Borel measurable function $F:\C\to [0,\infty)$, Borel set $\delta\subseteq \C$, sequence $\left\{\Delta_n\right\}_{n=1}^\infty$ 
of pairwise disjoint Borel sets in $\C$, $f\in X$, and $g^*\in X^*$,
\begin{equation}\label{decompose}
\int\limits_{\delta}F(\lambda)\,dv(E_A(\cup_{n=1}^\infty \Delta_n)f,g^*,\lambda)
=\sum_{n=1}^\infty \int\limits_{\delta\cap\Delta_n}F(\lambda)\,dv(E_A(\Delta_n)f,g^*,\lambda).
\end{equation}

\begin{rem}
Subsequently, the frequently used term {\it ``spectral measure"} is abbreviated to {\it s.m.}.
\end{rem} 

The following statement characterizing the domains of Borel measurable functions of a scalar type spectral operator in terms of Borel measures is fundamental for our discourse.

\begin{prop}[{\cite[Proposition $3.1$]{Markin2002(1)}}]\label{prop}\ \\
Let $A$ be a scalar type spectral operator in a complex Banach space $(X,\|\cdot\|)$ with spectral measure $E_A(\cdot)$ and $F:\sigma(A)\to \C$ be a Borel measurable function. Then $f\in D(F(A))$ iff
\begin{enumerate}[label={(\roman*)}]
\item for each $g^*\in X^*$, 
$\displaystyle \int\limits_{\sigma(A)} |F(\lambda)|\,d v(f,g^*,\lambda)<\infty$ and
\item $\displaystyle \sup_{\{g^*\in X^*\,|\,\|g^*\|=1\}}
\int\limits_{\{\lambda\in\sigma(A)\,|\,|F(\lambda)|>n\}}
|F(\lambda)|\,dv(f,g^*,\lambda)\to 0,\ n\to\infty$,
\end{enumerate}
where $v(f,g^*,\cdot)$ is the total variation measure of $\langle E_A(\cdot)f,g^* \rangle$.
\end{prop} 

The succeeding key theorem provides a description of the weak solutions of equation \eqref{1} with a scalar type spectral operator $A$ in a complex Banach space.

\begin{thm}[{\cite[Theorem $7$]{Markin2018(3)}}]\label{GWS}\ \\
Let $A$ be a scalar type spectral operator in a complex Banach space $(X,\|\cdot\|)$. A vector function $y:\R \to X$ is a weak solution 
of equation \eqref{1} iff there exists an $\displaystyle f \in \bigcap_{t\in \R}D(e^{tA})$ such that
\begin{equation}\label{expf}
y(t)=e^{tA}f,\ t\in \R,
\end{equation}
the operator exponentials understood in the sense of the Borel operational calculus (see \eqref{exp}).
\end{thm}

We also need the following characterization of a particular weak solution's of equation \eqref{1} with a scalar type spectral operator $A$ in a complex Banach space being strongly infinite differentiable on a subinterval $I$ of $\R$.

\begin{prop}[{\cite[Corollary $11$]{Markin2018(3)}}]\label{Cor}\ \\
Let $A$ be a scalar type spectral operator in a complex Banach space $(X,\|\cdot\|)$ and $I$ be interval of the real axis $\R$. A weak solution $y(\cdot)$ of equation \eqref{1} is strongly infinite differentiable on $I$ ($y(\cdot)\in C^\infty(I,X)$) iff, for each $t\in I$, 
\begin{equation*}
y(t) \in C^\infty(A):=\bigcap_{n=0}^{\infty}D(A^n),
\end{equation*}
in which case
\begin{equation*}
y^{(n)}(t)=A^ny(t),\ n\in \N,t\in I.
\end{equation*}
\end{prop}

\subsection{Gevrey Classes of Functions}\

\begin{defn}[Gevrey Classes of Functions]\ \\
Let $(X,\|\cdot\|)$ be a (real or complex) Banach space, $C^\infty(I,X)$ be the space of all $X$-valued functions strongly infinite differentiable on an interval $I$ of the real axis $\R$, and $0\le \beta<\infty$.

The subspaces
\[
\begin{split}
{\mathscr E}^{\{\beta\}}(I,X):=\bigl\{g(\cdot)\in C^{\infty}(I, X) \bigm |&
\forall [a,b] \subseteq I\ \exists \alpha>0\ \exists c>0:
\\
&\max_{a \le t \le b}\|g^{(n)}(t)\| \le c\alpha^n {(n!)}^\beta,
\ n\in\Z_+\bigr\}\\
\intertext{and}
{\mathscr E}^{(\beta)}(I,X):= \bigl\{g(\cdot) \in C^{\infty}(I,X) \bigm |& 
\forall [a,b] \subseteq I\ \forall \alpha > 0 \ \exists c>0:
\\
&\max_{a \le t \le b}\|g^{(n)}(t)\| \le c\alpha^n {(n!)}^\beta,
\ n\in\Z_+\bigr\}
\end{split}
\]
of $C^\infty(I,X)$ are called the {\it $\beta$th-order Gevrey classes} of strongly ultradifferentiable vector functions on $I$ of {\it Roumieu} and {\it Beurling type}, respectively (see, e.g., \cite{Gevrey,Komatsu1,Komatsu2,Komatsu3}).
\end{defn}

\begin{rems}\
\begin{itemize}
\item In view of {\it Stirling's formula}, the 
sequence $\left\{{(n!)}^\beta\right\}_{n=0}^\infty$ can be replaced with $\left\{ n^{\beta n}\right\}_{n=0}^\infty$.
\item For $0\le\beta<\beta'<\infty$, the inclusions
\begin{equation*}
{\mathscr E}^{(\beta)}(I,X)\subseteq{\mathscr E}^{\{\beta\}}(I,X)
\subseteq {\mathscr E}^{(\beta')}(I,X)\subseteq
{\mathscr E}^{\{\beta'\}}(I,X)\subseteq C^{\infty}(I,X)
\end{equation*}
hold.
\item For $1<\beta<\infty$, the Gevrey classes
${\mathscr E}^{(\beta)}(I,X)$ and ${\mathscr E}^{\{\beta\}}(I,X)$ are \textit{non-quasianalytic} (see, e.g., \cite{Komatsu2}).
\item The first-order Roumieu-type Gevrey class ${\mathscr E}^{\{1\}}(I,X)$ consists of all {\it analytic} on $I$, i.e., {\it analytically continuable} into complex neighborhoods of $I$, vector functions and the first-order Beurling-type Gevrey class ${\mathscr E}^{(1)}(I,X)$ consists of all {\it entire}, i.e., allowing {\it entire} continuations, vector functions \cite{Mandel}.
\item For $0\le\beta<1$, the Roumieu-type Gevrey class ${\mathscr E}^{\{\beta\}}(I,X)$ (the Beurling-type Gevrey class ${\mathscr E}^{(\beta)}(I,X)$) consists of all functions $g(\cdot)\in {\mathscr E}^{(1)}(I,X)$ such that, for some (any) $\gamma>0$, there exists an $M>0$, for which
\begin{equation}\label{order}
\|g(z)\|\le Me^{\gamma|z|^{1/(1-\beta)}},\ z\in \C,
\end{equation}
\cite{Markin2001(2)}. In particular,
for $\beta=0$, the Gevrey classes ${\mathscr E}^{\{0\}}(I,X)$ and ${\mathscr E}^{(0)}(I,X)$ are the classes of entire vector functions of \textit{exponential} and \textit{minimal exponential type}, respectively (see, e.g., \cite{Levin}).
\end{itemize} 
\end{rems} 

\subsection{Gevrey Classes of Vectors}\

\begin{defn}[Gevrey Classes of Vectors]\ \\
Let $A$ be a densely defined closed linear operator in a (real or complex) Banach space $(X,\|\cdot\|)$ and $0\le \beta<\infty$.

The following subspaces 
\[
\begin{split}
{\mathscr E}^{\{\beta\}}(A)&:=\left\{f\in C^{\infty}(A)\, \middle |\, 
\exists \alpha>0\ \exists c>0:
\|A^nf\| \le c\alpha^n {(n!)}^\beta,\ n\in\Z_+ \right\}\\
\intertext{and}
{\mathscr E}^{(\beta)}(A)&:=\left\{f \in C^{\infty}(A)\, \middle|\,\forall \alpha > 0 \ \exists c>0:
\|A^nf\| \le c\alpha^n {(n!)}^\beta,\ n\in\Z_+ \right\}
\end{split}
\]
of $C^{\infty}(A)$ are called the \textit{$\beta$th-order Gevrey classes} of ultradifferentiable vectors of $A$ of \textit{Roumieu} and \textit{Beurling type}, respectively (see, e.g., \cite{GorV83,Gor-Knyaz,book}).
\end{defn}

\begin{rems}\label{remsGCV}\
\begin{itemize}
\item In view of {\it Stirling's formula}, the 
sequence $\left\{{(n!)}^\beta\right\}_{n=0}^\infty$ can be replaced with $\left\{ n^{\beta n}\right\}_{n=0}^\infty$.
\item For $0\le\beta<\beta'<\infty$, the inclusions
\begin{equation*}
{\mathscr E}^{(\beta)}(A)\subseteq{\mathscr E}^{\{\beta\}}(A)
\subseteq {\mathscr E}^{(\beta')}(A)\subseteq
{\mathscr E}^{\{\beta'\}}(A)\subseteq C^{\infty}(A)
\end{equation*}
hold.
\item In particular, ${\mathscr E}^{\{1\}}(A)$ and ${\mathscr E}^{(1)}(A)$ are the classes of {\it analytic} and {\it entire} vectors of $A$, respectively \cite{Goodman,Nelson} and ${\mathscr E}^{\{0\}}(A)$ and ${\mathscr E}^{(0)}(A)$ are the classes of \textit{entire} vectors of $A$ of \textit{exponential} and \textit{minimal exponential type}, respectively (see, e.g., \cite{Radyno1983(1),Gor-Knyaz}).
\item As is readily seen, in view of the \textit{closedness} of $A$, the first-order Beurling-type Gevrey class ${\mathscr E}^{(1)}(A)$ forms the subspace of the initial values $f\in X$ generating the (classical) solutions of \eqref{1}, which are entire vector functions represented by the power series
\begin{equation*}
\sum_{n=0}^\infty \dfrac{t^n}{n!}A^nf,\ t\in \R,
\end{equation*}
the classes ${\mathscr E}^{\{\beta\}}(A)$ and ${\mathscr E}^{(\beta)}(A)$ with $0\le\beta<1$ being the subspaces of such initial values for which the solutions satisfy growth estimate \eqref{order} with some (any) $\gamma>0$ and some $M=M(\gamma)>0$, respectively (cf. \cite{Levin}).
\end{itemize} 
\end{rems} 

As is shown in \cite{GorV83} (see also \cite{Gor-Knyaz,book}), for a {\it normal operator} $A$ in a complex Hilbert space and any $0<\beta<\infty$,
\begin{equation}\label{GC}
{\mathscr E}^{\{\beta\}}(A)=\bigcup_{t>0} D(e^{t|A|^{1/\beta}})\quad \text{and}\quad 
{\mathscr E}^{(\beta)}(A)=\bigcap_{t>0} D(e^{t|A|^{1/\beta}}),
\end{equation}
the operator exponentials $e^{t|A|^{1/\beta}}$, $t>0$, understood in the sense of the Borel operational calculus (see, e.g., \cite{Dun-SchII,Plesner}).

In \cite{Markin2004(2),Markin2015}, descriptions \eqref{GC} are extended  to \textit{scalar type spectral operators} in a complex Banach space. In \cite{Markin2015}, similar nature descriptions of the classes ${\mathscr E}^{\{0\}}(A)$ and ${\mathscr E}^{(0)}(A)$ ($\beta=0$), known for a normal operator $A$ in a complex Hilbert space (see, e.g., \cite{Gor-Knyaz}), are also generalized to scalar type spectral operators in a complex Banach space. In particular {\cite[Theorem $5.1$]{Markin2015}},
\[
{\mathscr E}^{\{0\}}(A)=\bigcup_{\alpha>0}E_A(\Delta_\alpha)X,
\] 
where
\begin{equation*}
\Delta_\alpha:=\left\{\lambda\in\C\,\middle|\,|\lambda|\le \alpha \right\},\ \alpha>0.
\end{equation*}

\subsection{Gevrey Ultradifferentiability 
of a Particular Weak Solution of \eqref{+}}

We also need the following characterization of a particular weak solution's of equation \eqref{1} with a scalar type spectral operator $A$ in a complex Banach space being strongly Gevrey ultradifferentiable on a subinterval $I$ of $[0,\infty)$.

\begin{prop}[{\cite[Proposition $3.1$]{Markin2018(4)}}]\label{particular+}\ \\
Let $A$ be a scalar type spectral operator in a complex Banach space $(X,\|\cdot\|)$, $0\le \beta<\infty$, and $I$ be a subinterval of $[0,\infty)$. The restriction of
a weak solution $y(\cdot)$ of equation \eqref{+} to $I$ belongs to the Gevrey class ${\mathscr E}^{\{\beta\}}(I,X)$
\textup{(${\mathscr E}^{(\beta)}(I,X)$)} iff, for each $t\in I$,
\begin{equation*}
y(t) \in {\mathscr E}^{\{\beta\}}(A)
\ \textup{(${\mathscr E}^{(\beta)}(A)$, respectively)},
\end{equation*}
in which case
\begin{equation*}
y^{(n)}(t)=A^ny(t),\ n\in \N,t\in I.
\end{equation*}
\end{prop}

\section{Gevrey Ultradifferentiability 
of a Particular Weak Solution}

\begin{prop}[Gevrey Ultradifferentiability of a Particular Weak Solution]\label{particular}\ \\
Let $A$ be a scalar type spectral operator in a complex Banach space $(X,\|\cdot\|)$, $0\le \beta<\infty$, and $I$ be an interval of the real axis $\R$. The restriction of
a weak solution $y(\cdot)$ of equation \eqref{1} to $I$ belongs to the Gevrey class ${\mathscr E}^{\{\beta\}}(I,X)$
\textup{(${\mathscr E}^{(\beta)}(I,X)$)} iff, for each 
$t\in I$,
\begin{equation*}
y(t) \in {\mathscr E}^{\{\beta\}}(A)
\ \textup{(${\mathscr E}^{(\beta)}(A)$, respectively)},
\end{equation*}
in which case
\begin{equation*}
y^{(n)}(t)=A^ny(t),\ n\in \N,t\in I.
\end{equation*}
\end{prop}

\begin{proof}
As is noted in Remarks \ref{remsws}, $y:\R\to X$ is a weak solution of \eqref{1} \textit{iff} 
\[
y_+(t):=y(t),\ t\ge 0,
\]
is a weak solution of equation \eqref{+} and 
\[
y_-(t):=y(-t),\ t\ge 0,
\]
is a weak solution of equation \eqref{-}.

The statement immediately follows from Proposition \ref{particular+} applied to
\[
y_+(t):=y(t),\ t\ge 0,\quad \text{and}\quad y_-(t):=y(-t),\ t\ge 0,
\] 
for an arbitrary weak solution $y(\cdot)$ of equation \eqref{1} in view of
\[
{\mathscr E}^{\{\beta\}}(-A)={\mathscr E}^{\{\beta\}}(A)
\quad\text{and}\quad
{\mathscr E}^{(\beta)}(-A)={\mathscr E}^{(\beta)}(A).
\]
\end{proof}

\section{Gevrey Ultradifferentiability of order $\beta \ge 1$}

\begin{thm}[Gevrey Ultradifferentiability of order $\beta \ge 1$]\label{real}\ \\
Let $A$ be a scalar type spectral operator in a complex Banach space $(X,\|\cdot\|)$ with spectral measure $E_A(\cdot)$ and $ 1\le \beta<\infty$. Then the following statements are equivalent. 
\begin{enumerate}[label={(\roman*)}]
\item Every weak solution of equation \eqref{1} belongs to the $\beta$th-order Beurling-type Gevrey class 
${\mathscr E}^{(\beta )}\left(\R,X\right)$.
\item Every weak solution of equation \eqref{1} belongs to the $\beta$th-order Roumieu-type Gevrey class 
${\mathscr E}^{\{\beta \}}\left(\R,X\right)$.
\item There exist 
$b_+>0$ and $ b_->0$ such that the set $\sigma(A)\setminus {\mathscr P}^\beta_{b_-,b_+}$,
where
\begin{equation*}
{\mathscr P}^\beta_{b_-,b_+}:=\left\{\lambda \in \C\, \middle|\,
\Rep\lambda \le -b_-|\Imp\lambda|^{1/\beta} 
\ \text{or}\ 
\Rep\lambda \ge b_+|\Imp\lambda|^{1/\beta} \right\},
\end{equation*}
is bounded (see Figure \ref{fig:graph4}).
\end{enumerate}
\begin{figure}[h]
\centering
\includegraphics[height=2in]{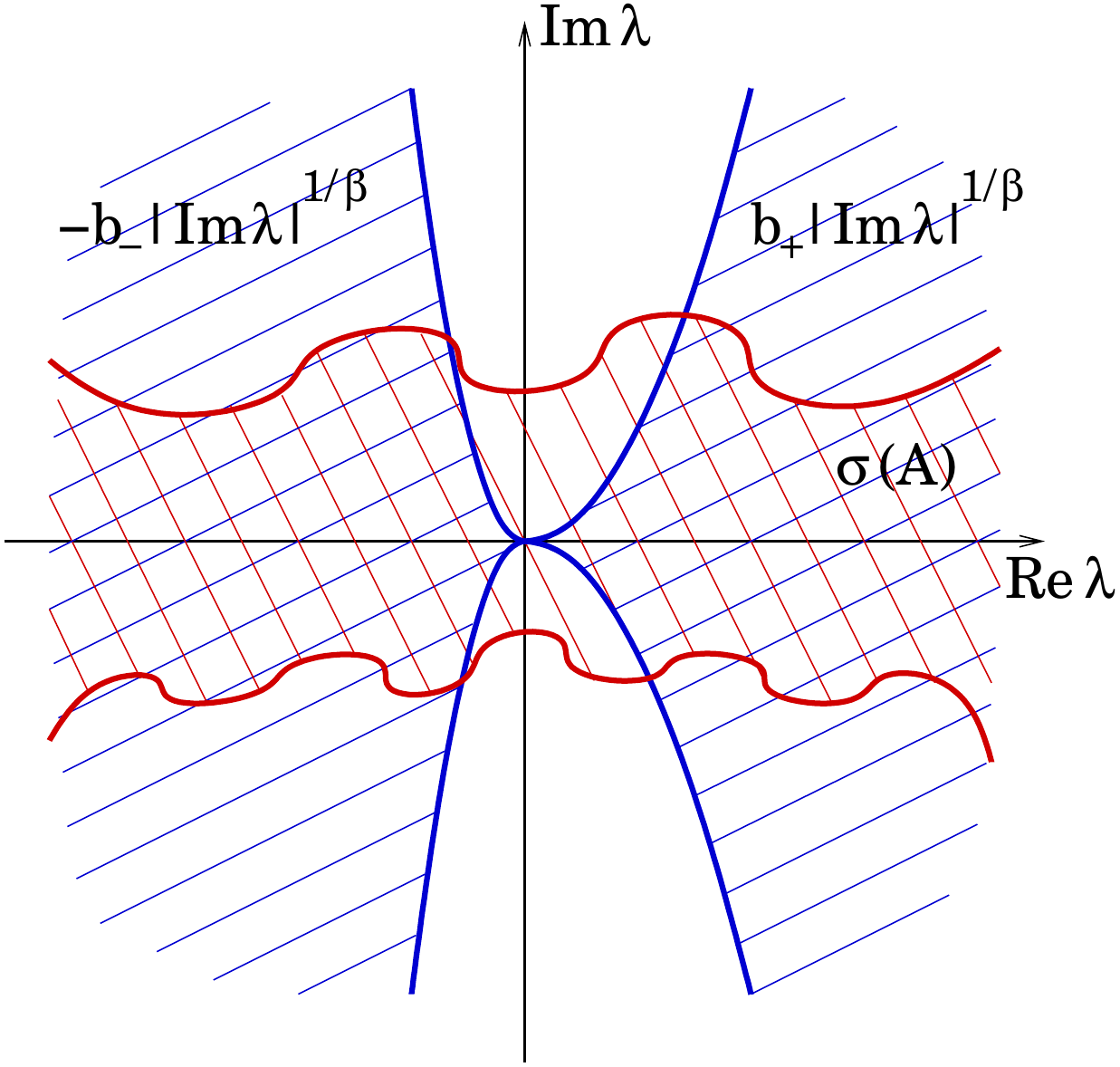}
\caption[]{Gevrey ultradifferentiability of order $1\le \beta<\infty$.}
\label{fig:graph4}
\end{figure}
\end{thm}

\begin{proof}
We are to prove the closed chain of implications
\begin{equation*}
\text{(i)}\Rightarrow \text{(ii)}\Rightarrow
\text{(iii)}\Rightarrow
\text{(i)},
\end{equation*} 
the implication $\text{(i)}\Rightarrow \text{(ii)}$
following immediately from the inclusion
\[
{\mathscr E}^{(\beta )}\left(\R,X\right)
\subseteq {\mathscr E}^{\{\beta\}}\left(\R,X\right)
\]
(see Remarks \ref{remsGCV}).

To prove the implication $\text{(iii)}\Rightarrow \text{(i)}$, suppose that there exist
$b_+>0$ and $ b_->0$ such that the set $\sigma(A)\setminus {\mathscr P}^\beta_{b_-,b_+}$
is \textit{bounded} and let $y(\cdot)$ be an arbitrary weak solution of equation \eqref{1}. 

By Theorem \ref{GWS}, 
\begin{equation*}
y(t)=e^{tA}f,\ t\in \R,\ \text{with some}\
f \in \bigcap_{t\in \R}D(e^{tA}).
\end{equation*}

Our purpose is to show that $y(\cdot)\in {\mathscr E}^{(\beta )}\left(\R,X\right)$, which, by Proposition \ref{particular} and \eqref{GC}, is accomplished by showing that, for each $t\in\R$,
\[
y(t)\in {\mathscr E}^{(\beta )}\left(A\right)
=\bigcap_{s>0} D(e^{s|A|^{1/\beta}}).
\]

Let us proceed by proving that, for any $t\in\R$ and $s>0$,
\[
y(t)\in D(e^{s|A|^{1/\beta}})
\] 
via Proposition \ref{prop}.

For any $s>0$, $t\in \R$ and an arbitrary $g^*\in X^*$,
\begin{multline}\label{first}
\int\limits_{\sigma(A)}e^{s|\lambda|^{1/\beta}}e^{t\Rep\lambda}\,dv(f,g^*,\lambda)
=\int\limits_{\sigma(A)\setminus{\mathscr P}_{b_-,b_+}^\beta}e^{s|\lambda|^{1/\beta}}e^{t\Rep\lambda}\,dv(f,g^*,\lambda)
\\
\shoveleft{
+\int\limits_{\left\{\lambda\in \sigma(A)\cap{\mathscr P}_{b_-,b_+}^\beta\,\middle|\,-1<\Rep\lambda<1 \right\}}e^{s|\lambda|^{1/\beta}}e^{t\Rep\lambda}\,dv(f,g^*,\lambda)
}\\
\shoveleft{
+\int\limits_{\left\{\lambda\in \sigma(A)\cap{\mathscr P}_{b_-,b_+}^\beta\,\middle|\,\Rep\lambda\ge 1 \right\}}e^{s|\lambda|^{1/\beta}}e^{t\Rep\lambda}\,dv(f,g^*,\lambda)
}\\
\hspace{1.2cm}
+\int\limits_{\left\{\lambda\in \sigma(A)\cap{\mathscr P}_{b_-,b_+}^\beta\,\middle|\,\Rep\lambda\le -1 \right\}}e^{s|\lambda|^{1/\beta}}e^{t\Rep\lambda}\,dv(f,g^*,\lambda)<\infty.
\hfill
\end{multline}

Indeed, 
\[
\int\limits_{\sigma(A)\setminus{\mathscr P}_{b_-,b_+}^\beta}e^{s|\lambda|^{1/\beta}}e^{t\Rep\lambda}\,dv(f,g^*,\lambda)<\infty
\]
and
\[
\int\limits_{\left\{\lambda\in \sigma(A)\cap{\mathscr P}_{b_-,b_+}^\beta\,\middle|\,-1<\Rep\lambda<1 \right\}}e^{s|\lambda|^{1/\beta}}e^{t\Rep\lambda}\,dv(f,g^*,\lambda)<\infty
\]
due to the boundedness of the sets
\[
\sigma(A)\setminus{\mathscr P}_{b_-,b_+}^\beta\ \text{and}\
\left\{\lambda\in \sigma(A)\cap{\mathscr P}_{b_-,b_+}^\beta\;\middle|\;-1<\Rep\lambda<1 \right\},
\]
the continuity of the integrated function on $\C$, and the finiteness of the measure $v(f,g^*,\cdot)$.

Further, for any $s>0$, $t\in\R$ and an arbitrary $g^*\in X^*$,
\begin{multline}\label{interm}
\int\limits_{\left\{\lambda\in \sigma(A)\cap{\mathscr P}_{b_-,b_+}^\beta\,\middle|\,\Rep\lambda\ge 1 \right\}}e^{s|\lambda|^{1/\beta}}e^{t\Rep\lambda}\,dv(f,g^*,\lambda)
\\
\shoveleft{
\le\int\limits_{\left\{\lambda\in \sigma(A)\cap{\mathscr P}_{b_-,b_+}^\beta\,\middle|\,\Rep\lambda\ge 1 \right\}}e^{s\left[|\Rep\lambda|+|\Imp\lambda|\right]^{1/\beta}}e^{t\Rep\lambda}\,dv(f,g^*,\lambda)
}\\
\hfill
\text{since, for $\lambda\in\sigma(A)\cap{\mathscr P}_{b_-,b_+}^\beta$ with $\Rep\lambda\ge 1$, $b_+^{-\beta}\Rep\lambda^\beta\ge |\Imp\lambda|$;}
\\
\shoveleft{
\le 
\int\limits_{\left\{\lambda\in \sigma(A)\cap{\mathscr P}_{b_-,b_+}^\beta\,\middle|\,\Rep\lambda\ge 1 \right\}}e^{s\left[\Rep\lambda+b_+^{-\beta}\Rep\lambda^\beta\right]^{1/\beta}}e^{t\Rep\lambda}\,dv(f,g^*,\lambda)
}\\
\hfill
\text{since, in view of $\Rep\lambda\ge 1$ and $\beta\ge 1$, $\Rep\lambda^\beta\ge\Rep\lambda$;}
\\
\shoveleft{
\le 
\int\limits_{\left\{\lambda\in \sigma(A)\cap{\mathscr P}_{b_-,b_+}^\beta\,\middle|\,\Rep\lambda\ge 1 \right\}}e^{s\left(1+b_+^{-\beta}\right)^{1/\beta}\Rep\lambda}e^{t\Rep\lambda}\,dv(f,g^*,\lambda)
}\\
\shoveleft{
= \int\limits_{\left\{\lambda\in \sigma(A)\cap{\mathscr P}_{b_-,b_+}^\beta\,\middle|\,\Rep\lambda\ge 1 \right\}}e^{\left[s\left(1+b_+^{-\beta}\right)^{1/\beta}+t\right]\Rep\lambda}\,dv(f,g^*,\lambda)
}\\
\hfill
\text{since $f\in \bigcap\limits_{t\in\R}D(e^{tA})$, by Proposition \ref{prop};}
\\
\hspace{1.2cm}
<\infty. 
\hfill
\end{multline}

Finally, for any $s>0$, $t\in\R$ and an arbitrary $g^*\in X^*$,
\begin{multline}\label{interm2}
\int\limits_{\left\{\lambda\in \sigma(A)\cap{\mathscr P}_{b_-,b_+}^\beta\,\middle|\,\Rep\lambda\le -1 \right\}}e^{s|\lambda|^{1/\beta}}e^{t\Rep\lambda}\,dv(f,g^*,\lambda)
\\
\shoveleft{
\le\int\limits_{\left\{\lambda\in \sigma(A)\cap{\mathscr P}_{b_-,b_+}^\beta\,\middle|\,\Rep\lambda\le -1 \right\}}e^{s\left[|\Rep\lambda|+|\Imp\lambda|\right]^{1/\beta}}e^{t\Rep\lambda}\,dv(f,g^*,\lambda)
}\\
\hfill
\text{since, for $\lambda\in\sigma(A)\cap{\mathscr P}_{b_-,b_+}^\beta$ with $\Rep\lambda\le -1$, $b_-^{-\beta}(-\Rep\lambda)^\beta\ge |\Imp\lambda|$;}
\\
\shoveleft{
\le 
\int\limits_{\left\{\lambda\in \sigma(A)\cap{\mathscr P}_{b_-,b_+}^\beta\,\middle|\,\Rep\lambda\le -1 \right\}}e^{s\left[-\Rep\lambda+b_-^{-\beta}(-\Rep\lambda)^\beta\right]^{1/\beta}}e^{t\Rep\lambda}\,dv(f,g^*,\lambda)
}\\
\hfill
\text{since, in view of $-\Rep\lambda\ge 1$ and $\beta\ge 1$, $(-\Rep\lambda)^\beta\ge-\Rep\lambda$;}
\\
\shoveleft{
\le 
\int\limits_{\left\{\lambda\in \sigma(A)\cap{\mathscr P}_{b_-,b_+}^\beta\,\middle|\,\Rep\lambda\le -1 \right\}}e^{s\left(1+b_-^{-\beta}\right)^{1/\beta}(-\Rep\lambda)}e^{t\Rep\lambda}\,dv(f,g^*,\lambda)
}\\
\shoveleft{
= \int\limits_{\left\{\lambda\in \sigma(A)\cap{\mathscr P}_{b_-,b_+}^\beta\,\middle|\,\Rep\lambda\le -1 \right\}}e^{\left[t-s\left(1+b_-^{-\beta}\right)^{1/\beta}\right]\Rep\lambda}\,dv(f,g^*,\lambda)
}\\
\hfill
\text{since $f\in \bigcap\limits_{t\in \R}D(e^{tA})$, by Proposition \ref{prop};}
\\
\hspace{1.2cm}
<\infty. 
\hfill
\end{multline}

Also, for any $s>0$, $t\in\R$ and an arbitrary $n\in\N$,
\begin{multline}\label{second}
\sup_{\{g^*\in X^*\,|\,\|g^*\|=1\}}
\int\limits_{\left\{\lambda\in\sigma(A)\,\middle|\,e^{s|\lambda|^{1/\beta}}e^{t\Rep\lambda}>n\right\}}
e^{s|\lambda|^{1/\beta}}e^{t\Rep\lambda}\,dv(f,g^*,\lambda)
\\
\shoveleft{
\le \sup_{\{g^*\in X^*\,|\,\|g^*\|=1\}}
\int\limits_{\left\{\lambda\in\sigma(A)\setminus{\mathscr P}_{b_-,b_+}^\beta\,\middle|\,e^{s|\lambda|^{1/\beta}}e^{t\Rep\lambda}>n\right\}}e^{s|\lambda|^{1/\beta}}e^{t\Rep\lambda}\,dv(f,g^*,\lambda)
}\\
\shoveleft{
+ \sup_{\{g^*\in X^*\,|\,\|g^*\|=1\}}
\int\limits_{\left\{\lambda\in\sigma(A)\cap{\mathscr P}_{b_-,b_+}^\beta\,\middle|\,-1<\Rep\lambda<1,\, e^{s|\lambda|^{1/\beta}}e^{t\Rep\lambda}>n\right\}}e^{s|\lambda|^{1/\beta}}e^{t\Rep\lambda}\,dv(f,g^*,\lambda)
}\\
\shoveleft{
+ \sup_{\{g^*\in X^*\,|\,\|g^*\|=1\}}
\int\limits_{\left\{\lambda\in\sigma(A)\cap{\mathscr P}_{b_-,b_+}^\beta\,\middle|\,\Rep\lambda\ge 1,\, e^{s|\lambda|^{1/\beta}}e^{t\Rep\lambda}>n\right\}}e^{s|\lambda|^{1/\beta}}e^{t\Rep\lambda}\,dv(f,g^*,\lambda)
}\\
\shoveleft{
+ \sup_{\{g^*\in X^*\,|\,\|g^*\|=1\}}
\int\limits_{\left\{\lambda\in\sigma(A)\cap{\mathscr P}_{b_-,b_+}^\beta\,\middle|\,\Rep\lambda\le -1,\, e^{s|\lambda|^{1/\beta}}e^{t\Rep\lambda}>n\right\}}e^{s|\lambda|^{1/\beta}}e^{t\Rep\lambda}\,dv(f,g^*,\lambda)
}\\
\hspace{1.2cm}
\to 0,\ n\to\infty.
\hfill
\end{multline}

Indeed, since, due to the boundedness of the sets
\[
\sigma(A)\setminus{\mathscr P}_{b_-,b_+}^\beta\ \text{and}\
\left\{\lambda\in\sigma(A)\cap{\mathscr P}_{b_-,b_+}^\beta\,\middle|\,-1<\Rep\lambda<1\right\}
\]
and the continuity of the integrated function 
on $\C$,
the sets
\[
\left\{\lambda\in\sigma(A)\setminus{\mathscr P}_{b_-,b_+}^\beta\,\middle|\,e^{s|\lambda|^{1/\beta}}e^{t\Rep\lambda}>n\right\}
\]
and 
\[
\left\{\lambda\in\sigma(A)\cap{\mathscr P}_{b_-,b_+}^\beta\,\middle|\,-1<\Rep\lambda<1,\, e^{s|\lambda|^{1/\beta}}e^{t\Rep\lambda}>n\right\}
\]
are \textit{empty} for all sufficiently large $n\in \N$,
we immediately infer that, for any $s>0$ and $t\in\R$,
\[
\lim_{n\to\infty}\sup_{\{g^*\in X^*\,|\,\|g^*\|=1\}}
\int\limits_{\left\{\lambda\in\sigma(A)\setminus{\mathscr P}_{b_-,b_+}^\beta\,\middle|\,e^{s|\lambda|^{1/\beta}}e^{t\Rep\lambda}>n\right\}}e^{s|\lambda|^{1/\beta}}e^{t\Rep\lambda}\,dv(f,g^*,\lambda)=0
\]
and
\[
\lim_{n\to\infty}\sup_{\{g^*\in X^*\,|\,\|g^*\|=1\}}
\int\limits_{\left\{\lambda\in\sigma(A)\cap{\mathscr P}_{b_-,b_+}^\beta\,\middle|\,-1<\Rep\lambda<1,\, e^{s|\lambda|^{1/\beta}}e^{t\Rep\lambda}>n\right\}}e^{s|\lambda|^{1/\beta}}e^{t\Rep\lambda}\,dv(f,g^*,\lambda)
=0.
\]

Further, for any $s>0$, $t\in\R$ and an arbitrary $n\in\N$,
\begin{multline*}
\sup_{\{g^*\in X^*\,|\,\|g^*\|=1\}}
\int\limits_{\left\{\lambda\in\sigma(A)\cap{\mathscr P}_{b_-,b_+}^\beta\,\middle|\,\Rep\lambda\ge 1,\, e^{s|\lambda|^{1/\beta}}e^{t\Rep\lambda}>n\right\}}e^{s|\lambda|^{1/\beta}}e^{t\Rep\lambda}\,dv(f,g^*,\lambda)
\\
\hfill
\text{as in \eqref{interm};}
\\
\shoveleft{
\le \sup_{\{g^*\in X^*\,|\,\|g^*\|=1\}}
\int\limits_{\left\{\lambda\in\sigma(A)\cap{\mathscr P}_{b_-,b_+}^\beta\,\middle|\,\Rep\lambda\ge 1,\, e^{s|\lambda|^{1/\beta}}e^{t\Rep\lambda}>n\right\}}e^{\left[s\left(1+b_+^{-\beta}\right)^{1/\beta}+t\right]\Rep\lambda}\,dv(f,g^*,\lambda)
}\\
\hfill
\text{since $f\in \bigcap\limits_{t\in \R}D(e^{tA})$, by \eqref{cond(ii)};}
\\
\shoveleft{
\le \sup_{\{g^*\in X^*\,|\,\|g^*\|=1\}}
}\\
\shoveleft{
4M\left\|E_A\left(\left\{\lambda\in\sigma(A)\cap{\mathscr P}_{b_-,b_+}^\beta\,\middle|\,\Rep\lambda\ge 1,\, e^{s|\lambda|^{1/\beta}}e^{t\Rep\lambda}>n\right\}\right)
e^{\left[s\left(1+b_+^{-\beta}\right)^{1/\beta}+t\right]A}f\right\|\|g^*\|
}\\
\shoveleft{
\le 4M\left\|E_A\left(\left\{\lambda\in\sigma(A)\cap{\mathscr P}_{b_-,b_+}^\beta\,\middle|\,\Rep\lambda\ge 1,\, e^{s|\lambda|^{1/\beta}}e^{t\Rep\lambda}>n\right\}\right)
e^{\left[s\left(1+b_+^{-\beta}\right)^{1/\beta}+t\right]A}f\right\|
}\\
\hfill
\text{by the strong continuity of the {\it s.m.};}
\\
\ \
\to 4M\left\|E_A\left(\emptyset\right)e^{\left[s\left(1+b_+^{-\beta}\right)^{1/\beta}+t\right]A}f\right\|=0,\ n\to\infty.
\hfill
\end{multline*}

Finally, for any $s>0$, $t\in\R$ and an arbitrary $n\in\N$,
\begin{multline*}
\sup_{\{g^*\in X^*\,|\,\|g^*\|=1\}}
\int\limits_{\left\{\lambda\in\sigma(A)\cap{\mathscr P}_{b_-,b_+}^\beta\,\middle|\,\Rep\lambda\le -1,\, e^{s|\lambda|^{1/\beta}}e^{t\Rep\lambda}>n\right\}}e^{s|\lambda|^{1/\beta}}e^{t\Rep\lambda}\,dv(f,g^*,\lambda)
\\
\hfill
\text{as in \eqref{interm2};}
\\
\shoveleft{
\le \sup_{\{g^*\in X^*\,|\,\|g^*\|=1\}}
\int\limits_{\left\{\lambda\in\sigma(A)\cap{\mathscr P}_{b_-,b_+}^\beta\,\middle|\,\Rep\lambda\le -1,\, e^{s|\lambda|^{1/\beta}}e^{t\Rep\lambda}>n\right\}}e^{\left[t-s\left(1+b_-^{-\beta}\right)^{1/\beta}\right]\Rep\lambda}\,dv(f,g^*,\lambda)
}\\
\hfill
\text{since $f\in \bigcap\limits_{t\in \R}D(e^{tA})$, by \eqref{cond(ii)};}
\\
\shoveleft{
\le \sup_{\{g^*\in X^*\,|\,\|g^*\|=1\}}
}\\
\shoveleft{
4M\left\|E_A\left(\left\{\lambda\in\sigma(A)\cap{\mathscr P}_{b_-,b_+}^\beta\,\middle|\,\Rep\lambda\ge 1,\, e^{s|\lambda|^{1/\beta}}e^{t\Rep\lambda}>n\right\}\right)
e^{\left[t-s\left(1+b_-^{-\beta}\right)^{1/\beta}\right]A}f\right\|\|g^*\|
}\\
\shoveleft{
\le 4M\left\|E_A\left(\left\{\lambda\in\sigma(A)\cap{\mathscr P}_{b_-,b_+}^\beta\,\middle|\,\Rep\lambda\ge 1,\, e^{s|\lambda|^{1/\beta}}e^{t\Rep\lambda}>n\right\}\right)
e^{\left[t-s\left(1+b_-^{-\beta}\right)^{1/\beta}\right]A}f\right\|
}\\
\hfill
\text{by the strong continuity of the {\it s.m.};}
\\
\ \
\to 4M\left\|E_A\left(\emptyset\right)e^{\left[t-s\left(1+b_-^{-\beta}\right)^{1/\beta}\right]A}f\right\|=0,\ n\to\infty.
\hfill
\end{multline*}

By Proposition \ref{prop} and the properties of the operational calculus (see {\cite[Theorem XVIII.$2.11$ (f)]{Dun-SchIII}}), \eqref{first} and \eqref{second} jointly imply that, for any $s>0$ and $t\in\R$,
\[
f\in D(e^{s|A|^{1/\beta}}e^{tA}),
\]
which, in view of \eqref{GC}, further implies that, for each $t\in\R$, 
\begin{equation*}
y(t)=e^{tA}f\in \bigcap_{s>0} D(e^{s|A|^{1/\beta}})
={\mathscr E}^{(\beta)}(A).
\end{equation*}

Whence, by Proposition \ref{particular}, we infer that
\begin{equation*}
y(\cdot) \in {\mathscr E}^{(\beta)}(\R,X),
\end{equation*}
which completes the proof for the implication $\text{(iii)}\Rightarrow \text{(i)}$.

\smallskip
Let us prove the remaining implication $\text{(ii)}\Rightarrow \text{(iii)}$ {\it by contrapositive} assuming that, for any $b_+>0$ and $b_->0$, the set 
$\sigma(A)\setminus {\mathscr P}_{b_-,b_+}^\beta$ is \textit{unbounded}. In particular, this means that, for any $n\in \N$, unbounded is the set
\begin{equation*}
\sigma(A)\setminus {\mathscr P}^\beta_{n^{-1},n^{-2}}=
\left\{\lambda \in \sigma(A)\,\middle| 
-n^{-2}|\Imp\lambda|^{1/\beta}<\Rep\lambda < n^{-2}|\Imp\lambda|^{1/\beta}\right\}.
\end{equation*} 

Hence, we can choose a sequence $\left\{\lambda_n\right\}_{n=1}^\infty$ of points in the complex plane as follows:
\begin{equation*}
\begin{split}
&\lambda_n \in \sigma(A),\ n\in \N,\\
&-n^{-2}|\Imp\lambda_n|^{1/\beta}<\Rep\lambda_n <n^{-2}|\Imp\lambda_n|^{1/\beta},\ n\in \N,\\
&\lambda_0:=0,\ |\lambda_n|>\max\left[n,|\lambda_{n-1}|\right],\ n\in \N.\\
\end{split}
\end{equation*}

The latter implies, in particular, that the points $\lambda_n$, $n\in\N$, are \textit{distinct} ($\lambda_i \neq \lambda_j$, $i\neq j$).

Since, for each $n\in \N$, the set
\begin{equation*}
\left\{ \lambda \in {\mathbb C}\,\middle|\, 
-n^{-2}|\Imp\lambda|^{1/\beta}<\Rep\lambda <n^{-2}|\Imp\lambda|^{1/\beta},\
|\lambda|>\max\bigl[n,|\lambda_{n-1}|\bigr]\right\}
\end{equation*}
is {\it open} in $\C$, along with the point $\lambda_n$, it contains an {\it open disk}
\begin{equation*}
\Delta_n:=\left\{\lambda \in \C\, \middle|\,|\lambda-\lambda_n|<\varepsilon_n \right\}
\end{equation*} 
centered at $\lambda_n$ of some radius $\varepsilon_n>0$, i.e., for each $\lambda \in \Delta_n$,
\begin{equation}\label{disks1}
-n^{-2}|\Imp\lambda|^{1/\beta}<\Rep\lambda < n^{-2}|\Imp\lambda|^{1/\beta}\ \text{and}\ |\lambda|>\max\bigl[n,|\lambda_{n-1}|\bigr].
\end{equation}

Furthermore, we can regard the radii of the disks to be small enough so that
\begin{equation}\label{radii1}
\begin{split}
&0<\varepsilon_n<\dfrac{1}{n},\ n\in\N,\ \text{and}\\
&\Delta_i \cap \Delta_j=\emptyset,\ i\neq j
\quad \text{(i.e., the disks are {\it pairwise disjoint})}.
\end{split}
\end{equation}

Whence, by the properties of the {\it s.m.}, 
\begin{equation*}
E_A(\Delta_i)E_A(\Delta_j)=0,\ i\neq j,
\end{equation*}
where $0$ stands for the \textit{zero operator} on $X$.

Observe also that the subspaces $E_A(\Delta_n)X$, $n\in \N$, are \textit{nontrivial} since
\[
\Delta_n \cap \sigma(A)\neq \emptyset,\ n\in\N,
\]
with $\Delta_n$ being an \textit{open set} in $\C$. 

In view of the pairwise disjointness of the disks $\Delta_n$, $n\in\N$, we can choose a unit vector $e_n\in E_A(\Delta_n)X$ for each $n\in\N$ obtaining a sequence 
$\left\{e_n\right\}_{n=1}^\infty$ in $X$ such that
\begin{equation}\label{ortho1}
\|e_n\|=1,\ n\in\N,\quad \text{and}\quad E_A(\Delta_i)e_j=\delta_{ij}e_j,\ i,j\in\N,
\end{equation}
where $\delta_{ij}$ is the \textit{Kronecker delta}.

As is easily seen, \eqref{ortho1} implies that the vectors $e_n$, $n\in \N$, are \textit{linearly independent}.

Furthermore, there exists an $\varepsilon>0$ such that
\begin{equation}\label{dist1}
d_n:=\dist\left(e_n,\spa\left(\left\{e_i\,|\,i\in\N,\ i\neq n\right\}\right)\right)\ge\varepsilon,\ n\in\N.
\end{equation}

Indeed, otherwise there exists a subsequence $\left\{d_{n(k)}\right\}_{k=1}^\infty$ such that
\begin{equation*}
d_{n(k)}\to 0,\ k\to\infty.
\end{equation*}

Then, by selecting a vector
\[
f_{n(k)}\in 
\spa\left(\left\{e_i\,|\,i\in\N,\ i\neq n(k)\right\}\right),\ k\in\N,
\] 
such that 
\[
\|e_{n(k)}-f_{n(k)}\|<d_{n(k)}+1/k,\ k\in\N,
\]
we arrive at
\begin{multline*}
1=\|e_{n(k)}\|
\\
\hfill
\text{since, by \eqref{ortho1}, $E_A(\Delta_{n(k)})e_{n(k)}=e_{n(k)}$ and 
$E_A(\Delta_{n(k)})f_{n(k)}=0$;}
\\
\shoveleft{
=\|E_A(\Delta_{n(k)})(e_{n(k)}-f_{n(k)})\|\
\le \|E_A(\Delta_{n(k)})\|\|e_{n(k)}-f_{n(k)}\|
\hfill
\text{by \eqref{bounded};}
}\\
\ \
\le M\|e_{n(k)}-f_{n(k)}\|\le M\left[d_{n(k)}+1/k\right]
\to 0,\ k\to\infty,
\hfill
\end{multline*}
which is a \textit{contradiction} proving \eqref{dist1}. 

As follows from the {\it Hahn-Banach Theorem}, for any $n\in\N$, there is an $e^*_n\in X^*$ such that 
\begin{equation}\label{H-B1}
\|e_n^*\|=1,\ n\in\N,\ \text{and}\ \langle e_i,e_j^*\rangle=\delta_{ij}d_i,\ i,j\in\N.
\end{equation}

Let us consider separately the two possibilities concerning the sequence of the real parts $\{\Rep\lambda_n\}_{n=1}^\infty$: its being \textit{bounded} or \textit{unbounded}. 

First, suppose that the sequence $\{\Rep\lambda_n\}_{n=1}^\infty$ is \textit{bounded}, i.e., there is such an $\omega>0$ that
\begin{equation}\label{bounded1}
|\Rep\lambda_n| \le \omega,\ n\in\N,
\end{equation}
and consider the element
\begin{equation*}
f:=\sum_{k=1}^\infty k^{-2}e_k\in X,
\end{equation*}
which is well defined since $\left\{k^{-2}\right\}_{k=1}^\infty\in l_1$ ($l_1$ is the space of absolutely summable sequences) and $\|e_k\|=1$, $k\in\N$ (see \eqref{ortho1}).

In view of \eqref{ortho1}, by the properties of the \textit{s.m.},
\begin{equation}\label{vectors1}
E_A(\cup_{k=1}^\infty\Delta_k)f=f\quad \text{and}\quad E_A(\Delta_k)f=k^{-2}e_k,\ k\in\N.
\end{equation}

For any $t\ge 0$ and an arbitrary $g^*\in X^*$,
\begin{multline}\label{first1}
\int\limits_{\sigma(A)}e^{t\Rep\lambda}\,dv(f,g^*,\lambda)
\hfill \text{by \eqref{vectors1};}
\\
\shoveleft{
=\int\limits_{\sigma(A)} e^{t\Rep\lambda}\,d v(E_A(\cup_{k=1}^\infty \Delta_k)f,g^*,\lambda)
\hfill
\text{by \eqref{decompose};}
}\\
\shoveleft{
=\sum_{k=1}^\infty\int\limits_{\sigma(A)\cap\Delta_k}e^{t\Rep\lambda}\,dv(E_A(\Delta_k)f,g^*,\lambda)
\hfill 
\text{by \eqref{vectors1};}
}\\
\shoveleft{
=\sum_{k=1}^\infty k^{-2}\int\limits_{\sigma(A)\cap\Delta_k}e^{t\Rep\lambda}\,dv(e_k,g^*,\lambda)
}\\
\hfill
\text{since, for $\lambda\in \Delta_k$, by \eqref{bounded1} and \eqref{radii1},}\ 
\Rep\lambda=\Rep\lambda_k+(\Rep\lambda-\Rep\lambda_k)
\\
\hfill
\le \Rep\lambda_k+|\lambda-\lambda_k|\le \omega+\varepsilon_k\le \omega+1;
\\
\shoveleft{
\le e^{t(\omega+1)}\sum_{k=1}^\infty k^{-2}\int\limits_{\sigma(A)\cap\Delta_k}1\,dv(e_k,g^*,\lambda)
= e^{t(\omega+1)}\sum_{k=1}^\infty k^{-2}v(e_k,g^*,\Delta_k)
}\\
\hfill
\text{by \eqref{tv};}
\\
\hspace{1.2cm}
\le e^{t(\omega+1)}\sum_{k=1}^\infty k^{-2}4M\|e_k\|\|g^*\|
= 4Me^{t(\omega+1)}\|g^*\|\sum_{k=1}^\infty k^{-2}<\infty.
\hfill
\end{multline} 

Also, for any $t<0$ and an arbitrary $g^*\in X^*$,
\begin{multline}\label{ffirst1}
\int\limits_{\sigma(A)}e^{t\Rep\lambda}\,dv(f,g^*,\lambda)
\hfill \text{by \eqref{vectors1};}
\\
\shoveleft{
=\int\limits_{\sigma(A)} e^{t\Rep\lambda}\,d v(E_A(\cup_{k=1}^\infty \Delta_k)f,g^*,\lambda)
\hfill
\text{by \eqref{decompose};}
}\\
\shoveleft{
=\sum_{k=1}^\infty\int\limits_{\sigma(A)\cap\Delta_k}e^{t\Rep\lambda}\,dv(E_A(\Delta_k)f,g^*,\lambda)
\hfill 
\text{by \eqref{vectors1};}
}\\
\shoveleft{
=\sum_{k=1}^\infty k^{-2}\int\limits_{\sigma(A)\cap\Delta_k}e^{t\Rep\lambda}\,dv(e_k,g^*,\lambda)
}\\
\hfill
\text{since, for $\lambda\in \Delta_k$, by \eqref{bounded1} and \eqref{radii1},}\ 
\Rep\lambda=\Rep\lambda_k-(\Rep\lambda_k-\Rep\lambda)
\\
\hfill
\ge \Rep\lambda_k-|\Rep\lambda_k-\Rep\lambda|\ge -\omega-\varepsilon_k\ge -\omega-1;
\\
\shoveleft{
\le e^{-t(\omega+1)}\sum_{k=1}^\infty k^{-2}\int\limits_{\sigma(A)\cap\Delta_k}1\,dv(e_k,g^*,\lambda)
= e^{-t(\omega+1)}\sum_{k=1}^\infty k^{-2}v(e_k,g^*,\Delta_k)
}\\
\hfill
\text{by \eqref{tv};}
\\
\hspace{1.2cm}
\le e^{-t(\omega+1)}\sum_{k=1}^\infty k^{-2}4M\|e_k\|\|g^*\|
= 4Me^{-t(\omega+1)}\|g^*\|\sum_{k=1}^\infty k^{-2}<\infty.
\hfill
\end{multline} 

Similarly, to \eqref{first1} for any $t\ge 0$ and an arbitrary $n\in\N$,
\begin{multline}\label{second1}
\sup_{\{g^*\in X^*\,|\,\|g^*\|=1\}}
\int\limits_{\left\{\lambda\in\sigma(A)\,\middle|\,e^{t\Rep\lambda}>n\right\}} 
e^{t\Rep\lambda}\,dv(f,g^*,\lambda)
\\
\shoveleft{
\le 
\sup_{\{g^*\in X^*\,|\,\|g^*\|=1\}}e^{t(\omega+1)}\sum_{k=1}^\infty k^{-2}
\int\limits_{\left\{\lambda\in\sigma(A)\,\middle|\,e^{t\Rep\lambda}>n\right\}\cap \Delta_k}1\,dv(e_k,g^*,\lambda) 
}\\
\hfill \text{by \eqref{vectors1};}
\\
\shoveleft{
=e^{t(\omega+1)}\sup_{\{g^*\in X^*\,|\,\|g^*\|=1\}}\sum_{k=1}^\infty 
\int\limits_{\left\{\lambda\in\sigma(A)\,\middle|\,e^{t\Rep\lambda}>n\right\}\cap \Delta_k}1\,dv(E_A(\Delta_k)f,g^*,\lambda) 
}\\
\hfill \text{by \eqref{decompose};}
\\
\shoveleft{
= e^{t(\omega+1)}\sup_{\{g^*\in X^*\,|\,\|g^*\|=1\}}
\int\limits_{\{\lambda\in\sigma(A)\,|\,e^{t\Rep\lambda}>n\}}1\,dv(E_A(\cup_{k=1}^\infty\Delta_k)f,g^*,\lambda)
}\\
\hfill \text{by \eqref{vectors1};}
\\
\shoveleft{
= e^{t(\omega+1)}\sup_{\{g^*\in X^*\,|\,\|g^*\|=1\}}
\int\limits_{\{\lambda\in\sigma(A)\,|\,e^{t\Rep\lambda}>n\}}1\,dv(f,g^*,\lambda)
\hfill
\text{by \eqref{cond(ii)};}
}\\
\shoveleft{
\le e^{t(\omega+1)}\sup_{\{g^*\in X^*\,|\,\|g^*\|=1\}}4M\left\|E_A\left(\left\{\lambda\in\sigma(A)\,\middle|\,e^{t\Rep\lambda}>n\right\}\right)f\right\|\|g^*\|
}\\
\shoveleft{
\le 4Me^{t(\omega+1)}\left\|E_A\left(\left\{\lambda\in\sigma(A)\,\middle|\,e^{t\Rep\lambda}>n\right\}\right)f\right\|
}\\
\hfill
\text{by the strong continuity of the {\it s.m.};}
\\
\hspace{1.2cm}
\to 4Me^{t(\omega+1)}\left\|E_A\left(\emptyset\right)f\right\|=0,\ n\to\infty.
\hfill
\end{multline}

Similarly, to \eqref{ffirst1} for any $t<0$ and an arbitrary $n\in\N$,
\begin{multline}\label{ssecond1}
\sup_{\{g^*\in X^*\,|\,\|g^*\|=1\}}
\int\limits_{\left\{\lambda\in\sigma(A)\,\middle|\,e^{t\Rep\lambda}>n\right\}} 
e^{t\Rep\lambda}\,dv(f,g^*,\lambda)
\\
\shoveleft{
\le 
\sup_{\{g^*\in X^*\,|\,\|g^*\|=1\}}e^{-t(\omega+1)}\sum_{k=1}^\infty k^{-2}
\int\limits_{\left\{\lambda\in\sigma(A)\,\middle|\,e^{t\Rep\lambda}>n\right\}\cap \Delta_k}1\,dv(e_k,g^*,\lambda) 
}\\
\hfill \text{by \eqref{vectors1};}
\\
\shoveleft{
=e^{-t(\omega+1)}\sup_{\{g^*\in X^*\,|\,\|g^*\|=1\}}\sum_{k=1}^\infty 
\int\limits_{\left\{\lambda\in\sigma(A)\,\middle|\,e^{t\Rep\lambda}>n\right\}\cap \Delta_k}1\,dv(E_A(\Delta_k)f,g^*,\lambda) 
}\\
\hfill \text{by \eqref{decompose};}
\\
\shoveleft{
= e^{-t(\omega+1)}\sup_{\{g^*\in X^*\,|\,\|g^*\|=1\}}
\int\limits_{\{\lambda\in\sigma(A)\,|\,e^{t\Rep\lambda}>n\}}1\,dv(E_A(\cup_{k=1}^\infty\Delta_k)f,g^*,\lambda)
}\\
\hfill \text{by \eqref{vectors1};}
\\
\shoveleft{
= e^{-t(\omega+1)}\sup_{\{g^*\in X^*\,|\,\|g^*\|=1\}}
\int\limits_{\{\lambda\in\sigma(A)\,|\,e^{t\Rep\lambda}>n\}}1\,dv(f,g^*,\lambda)
\hfill
\text{by \eqref{cond(ii)};}
}\\
\shoveleft{
\le e^{-t(\omega+1)}\sup_{\{g^*\in X^*\,|\,\|g^*\|=1\}}4M\left\|E_A\left(\left\{\lambda\in\sigma(A)\,\middle|\,e^{t\Rep\lambda}>n\right\}\right)f\right\|\|g^*\|
}\\
\shoveleft{
\le 4Me^{-t(\omega+1)}\left\|E_A\left(\left\{\lambda\in\sigma(A)\,\middle|\,e^{t\Rep\lambda}>n\right\}\right)f\right\|
}\\
\hfill
\text{by the strong continuity of the {\it s.m.};}
\\
\hspace{1.2cm}
\to 4Me^{-t(\omega+1)}\left\|E_A\left(\emptyset\right)f\right\|=0,\ n\to\infty.
\hfill
\end{multline}

By Proposition \ref{prop}, \eqref{first1}, \eqref{ffirst1}, \eqref{second1}, and \eqref{ssecond1} jointly imply that 
\[
f\in \bigcap\limits_{t\in\R}D(e^{tA}),
\]
and hence, by Theorem \ref{GWS},
\[
y(t):=e^{tA}f,\ t\in\R,
\]
is a weak solution of equation \eqref{1}.

Let
\begin{equation}\label{functional1}
h^*:=\sum_{k=1}^\infty k^{-2}e_k^*\in X^*,
\end{equation}
the functional being well defined since $\{k^{-2}\}_{k=1}^\infty\in l_1$ and $\|e_k^*\|=1$, $k\in\N$ (see \eqref{H-B1}).

In view of \eqref{H-B1} and \eqref{dist1}, we have:
\begin{equation}\label{funct-dist1}
\langle e_n,h^*\rangle=\langle e_k,k^{-2}e_k^*\rangle=d_k k^{-2}\ge \varepsilon k^{-2},\ k\in\N.
\end{equation}

For any $s>0$,
\begin{multline*}
\int\limits_{\sigma(A)}e^{s|\lambda|^{1/\beta}}\,dv(f,h^*,\lambda)
\hfill
\text{by \eqref{decompose} as in \eqref{first1};}
\\
\shoveleft{
=\sum_{k=1}^\infty k^{-2}\int\limits_{\sigma(A)\cap\Delta_k}e^{s|\lambda|^{1/\beta}}\,dv(e_k,h^*,\lambda)
\hfill
\text{since, for $\lambda\in \Delta_k$, by \eqref{disks1}, $|\lambda|\ge k$;}
}\\
\shoveleft{
\ge
\sum_{k=1}^\infty k^{-2}e^{sk^{1/\beta}}\int\limits_{\sigma(A)\cap\Delta_k}1\,dv(e_k,h^*,\lambda)
= \sum_{k=1}^\infty k^{-2}e^{sk^{1/\beta}} v(e_k,h^*,\Delta_k)
}\\
\shoveleft{
\ge\sum_{k=1}^\infty k^{-2}e^{sk^{1/\beta}}|\langle E_A(\Delta_k)e_k,h^*\rangle|
\hfill
\text{by \eqref{ortho1} and \eqref{funct-dist1};}
}\\
\ \
\ge \sum_{k=1}^\infty \varepsilon k^{-4}e^{sk^{1/\beta}}=\infty.
\hfill
\end{multline*} 

Whence, by Proposition \ref{prop} and \eqref{GC}, we infer that
\[
y(0)=f\notin \bigcup_{s>0} D(e^{s|A|^{1/\beta}})
={\mathscr E}^{\{\beta\}}(A)
\]
which, by Proposition \ref{particular}, implies that the weak solution $y(t)=e^{tA}f$, $t\in\R$, 
of equation \eqref{1} does not belong to the Roumieu-type Gevrey class ${\mathscr E}^{\{\beta \}}\left( \R,X\right)$ and completes our consideration of the case of the sequence's $\{\Rep\lambda_n\}_{n=1}^\infty$ being \textit{bounded}. 

Now, suppose that the sequence $\{\Rep\lambda_n\}_{n=1}^\infty$
is \textit{unbounded}. 

Therefore, there is a subsequence $\{\Rep\lambda_{n(k)}\}_{k=1}^\infty$ such that
\[
\Rep\lambda_{n(k)}\to \infty \ \text{or}\ \Rep\lambda_{n(k)}\to -\infty,\ k\to \infty.
\]
Let us consider separately each of the two cases.

First, suppose that 
\[
\Rep\lambda_{n(k)}\to \infty,\ k\to \infty
\] 
Then, without loss of generality, we can regard that
\begin{equation}\label{infinity}
\Rep\lambda_{n(k)} \ge k,\ k\in\N.
\end{equation}

Consider the elements
\begin{equation*}
f:=\sum_{k=1}^\infty e^{-n(k)\Rep\lambda_{n(k)}}e_{n(k)}\in X
\ \text{and}\ h:=\sum_{k=1}^\infty e^{-\frac{n(k)}{2}\Rep\lambda_{n(k)}}e_{n(k)}\in X,
\end{equation*}
well defined since, by \eqref{infinity},
\[
\left\{e^{-n(k)\Rep\lambda_{n(k)}}\right\}_{k=1}^\infty,
\left\{e^{-\frac{n(k)}{2}\Rep\lambda_{n(k)}}\right\}_{k=1}^\infty
\in l_1
\]
and $\|e_{n(k)}\|=1$, $k\in\N$ (see \eqref{ortho1}).

By \eqref{ortho1},
\begin{equation}\label{subvectors1}
E_A(\cup_{k=1}^\infty\Delta_{n(k)})f=f\ \text{and}\
E_A(\Delta_{n(k)})f=e^{-n(k)\Rep\lambda_{n(k)}}e_{n(k)},\
k\in\N,
\end{equation}
and
\begin{equation}\label{subvectors12}
E_A(\cup_{k=1}^\infty\Delta_{n(k)})h=h\ \text{and}\
E_A(\Delta_{n(k)})h=e^{-\frac{n(k)}{2}\Rep\lambda_{n(k)}}e_{n(k)},\ k\in\N.
\end{equation}

For any $t\ge 0$ and an arbitrary $g^*\in X^*$, 
\begin{multline}\label{first2}
\int\limits_{\sigma(A)}e^{t\Rep\lambda}\,dv(f,g^*,\lambda)
\hfill
\text{by \eqref{decompose} as in \eqref{first1};}
\\
\shoveleft{
=\sum_{k=1}^\infty e^{-n(k)\Rep\lambda_{n(k)}}\int\limits_{\sigma(A)\cap\Delta_{n(k)}}e^{t\Rep\lambda}\,dv(e_{n(k)},g^*,\lambda)
}\\
\hfill
\text{since, for $\lambda\in \Delta_{n(k)}$, by \eqref{radii1},}\ \Rep\lambda
=\Rep\lambda_{n(k)}+(\Rep\lambda-\Rep\lambda_{n(k)})
\\
\hfill
\le \Rep\lambda_{n(k)}+|\lambda-\lambda_{n(k)}|\le \Rep\lambda_{n(k)}+1;
\\
\shoveleft{
\le \sum_{k=1}^\infty e^{-n(k)\Rep\lambda_{n(k)}}
e^{t(\Rep\lambda_{n(k)}+1)}
\int\limits_{\sigma(A)\cap\Delta_{n(k)}}1\,dv(e_{n(k)},g^*,\lambda)
}\\
\shoveleft{
= e^t\sum_{k=1}^\infty e^{-[n(k)-t]\Rep\lambda_{n(k)}}v(e_{n(k)},g^*,\Delta_{n(k)})
\hfill
\text{by \eqref{tv};}
}\\
\shoveleft{
\le e^t\sum_{k=1}^\infty e^{-[n(k)-t]\Rep\lambda_{n(k)}}4M\|e_{n(k)}\|\|g^*\|
= 4Me^t\|g^*\|\sum_{k=1}^\infty e^{-[n(k)-t]\Rep\lambda_{n(k)}}
}\\
\hspace{1.2cm}
<\infty.
\hfill
\end{multline}

Indeed, for all $k\in \N$ sufficiently large so that
\[
n(k)\ge t+1,
\]
in view of \eqref{infinity}, 
\[
e^{-[n(k)-t]\Rep\lambda_{n(k)}}\le e^{-k}.
\]

For any $t<0$ and an arbitrary $g^*\in X^*$, 
\begin{multline}\label{ffirst2}
\int\limits_{\sigma(A)}e^{t\Rep\lambda}\,dv(f,g^*,\lambda)
\hfill
\text{by \eqref{decompose} as in \eqref{first1};}
\\
\shoveleft{
=\sum_{k=1}^\infty e^{-n(k)\Rep\lambda_{n(k)}}\int\limits_{\sigma(A)\cap\Delta_{n(k)}}e^{t\Rep\lambda}\,dv(e_{n(k)},g^*,\lambda)
}\\
\hfill
\text{since, for $\lambda\in \Delta_{n(k)}$, by \eqref{radii1},}\ \Rep\lambda
=\Rep\lambda_{n(k)}-(\Rep\lambda_{n(k)}-\Rep\lambda)
\\
\hfill
\ge \Rep\lambda_{n(k)}-|\Rep\lambda_{n(k)}-\Rep\lambda|\ge \Rep\lambda_{n(k)}-1;
\\
\shoveleft{
\le \sum_{k=1}^\infty e^{-n(k)\Rep\lambda_{n(k)}}
e^{t(\Rep\lambda_{n(k)}-1)}
\int\limits_{\sigma(A)\cap\Delta_{n(k)}}1\,dv(e_{n(k)},g^*,\lambda)
}\\
\shoveleft{
= e^{-t}\sum_{k=1}^\infty e^{-[n(k)-t]\Rep\lambda_{n(k)}}v(e_{n(k)},g^*,\Delta_{n(k)})
\hfill
\text{by \eqref{tv};}
}\\
\shoveleft{
\le e^{-t}\sum_{k=1}^\infty e^{-[n(k)-t]\Rep\lambda_{n(k)}}4M\|e_{n(k)}\|\|g^*\|
= 4Me^{-t}\|g^*\|\sum_{k=1}^\infty e^{-[n(k)-t]\Rep\lambda_{n(k)}}
}\\
\hspace{1.2cm}
<\infty.
\hfill
\end{multline}

Indeed, for all $k\in \N$, in view of $t<0$,
\[
n(k)-t\ge n(k)\ge 1,
\]
and hence, in view of \eqref{infinity}, 
\[
e^{-[n(k)-t]\Rep\lambda_{n(k)}}\le e^{-k}.
\]

Similarly to \eqref{first2}, for any $t\ge 0$ and an arbitrary $n\in\N$,
\begin{multline}\label{second2}
\sup_{\{g^*\in X^*\,|\,\|g^*\|=1\}}
\int\limits_{\left\{\lambda\in\sigma(A)\,\middle|\,e^{t\Rep\lambda}>n\right\}}e^{t\Rep\lambda}\,dv(f,g^*,\lambda)
\\
\shoveleft{
\le \sup_{\{g^*\in X^*\,|\,\|g^*\|=1\}}e^t\sum_{k=1}^\infty e^{-[n(k)-t]\Rep\lambda_{n(k)}}
\int\limits_{\left\{\lambda\in\sigma(A)\,\middle|\,e^{t\Rep\lambda}>n\right\}\cap \Delta_{n(k)}}1\,dv(e_{n(k)},g^*,\lambda)
}\\
\shoveleft{
=e^t\sup_{\{g^*\in X^*\,|\,\|g^*\|=1\}}\sum_{k=1}^\infty e^{-\left[\frac{n(k)}{2}-t\right]\Rep\lambda_{n(k)}}
e^{-\frac{n(k)}{2}\Rep\lambda_{(k)}}
}\\
\shoveleft{
\int\limits_{\left\{\lambda\in\sigma(A)\,\middle|\,e^{t\Rep\lambda}>n\right\}\cap \Delta_{n(k)}}1\,dv(e_{n(k)},g^*,\lambda)
}\\
\hfill
\text{since, by \eqref{infinity}, there is an $L>0$ such that
$e^{-\left[\frac{n(k)}{2}-t\right]\Rep\lambda_{n(k)}}\le L$, $k\in\N$;}
\\
\shoveleft{
\le Le^t\sup_{\{g^*\in X^*\,|\,\|g^*\|=1\}}\sum_{k=1}^\infty e^{-\frac{n(k)}{2}\Rep\lambda_{n(k)}}
\int\limits_{\left\{\lambda\in\sigma(A)\,\middle|\,e^{t\Rep\lambda}>n\right\}\cap \Delta_{n(k)}}1\,dv(e_{n(k)},g^*,\lambda)
}\\
\hfill
\text{by \eqref{subvectors12};}
\\
\shoveleft{
= Le^t\sup_{\{g^*\in X^*\,|\,\|g^*\|=1\}}\sum_{k=1}^\infty
\int\limits_{\left\{\lambda\in\sigma(A)\,\middle|\,e^{t\Rep\lambda}>n\right\}\cap \Delta_{n(k)}}1\,dv(E_A(\Delta_{n(k)})h,g^*,\lambda)
}\\
\hfill
\text{by \eqref{decompose};}
\\
\shoveleft{
= Le^t\sup_{\{g^*\in X^*\,|\,\|g^*\|=1\}}
\int\limits_{\left\{\lambda\in\sigma(A)\,\middle|\,e^{t\Rep\lambda}>n\right\}}1\,dv(E_A(\cup_{k=1}^\infty\Delta_{n(k)})h,g^*,\lambda)
}\\
\hfill
\text{by \eqref{subvectors12};}
\\
\shoveleft{
=Le^t\sup_{\{g^*\in X^*\,|\,\|g^*\|=1\}}\int\limits_{\{\lambda\in\sigma(A)\,|\,e^{t\Rep\lambda}>n\}}1\,dv(h,g^*,\lambda)
\hfill
\text{by \eqref{cond(ii)};}
}\\
\shoveleft{
\le Le^t\sup_{\{g^*\in X^*\,|\,\|g^*\|=1\}}4M
\left\|E_A\left(\left\{\lambda\in\sigma(A)\,\middle|\,e^{t\Rep\lambda}>n\right\}\right)h\right\|\|g^*\|
}\\
\shoveleft{
\le 4LMe^t\|E_A(\{\lambda\in\sigma(A)\,|\,e^{t\Rep\lambda}>n\})h\|
}\\
\hfill
\text{by the strong continuity of the {\it s.m.};}
\\
\hspace{1.2cm}
\to 4LMe^t\left\|E_A\left(\emptyset\right)h\right\|=0,\ n\to\infty.
\hfill
\end{multline}

Similarly to \eqref{ffirst2}, for any $t<0$ and an arbitrary $n\in\N$,
\begin{multline}\label{ssecond2}
\sup_{\{g^*\in X^*\,|\,\|g^*\|=1\}}
\int\limits_{\left\{\lambda\in\sigma(A)\,\middle|\,e^{t\Rep\lambda}>n\right\}}e^{t\Rep\lambda}\,dv(f,g^*,\lambda)
\\
\shoveleft{
\le \sup_{\{g^*\in X^*\,|\,\|g^*\|=1\}}e^{-t}\sum_{k=1}^\infty e^{-[n(k)-t]\Rep\lambda_{n(k)}}
\int\limits_{\left\{\lambda\in\sigma(A)\,\middle|\,e^{t\Rep\lambda}>n\right\}\cap \Delta_{n(k)}}1\,dv(e_{n(k)},g^*,\lambda)
}\\
\shoveleft{
=e^{-t}\sup_{\{g^*\in X^*\,|\,\|g^*\|=1\}}\sum_{k=1}^\infty e^{-\left[\frac{n(k)}{2}-t\right]\Rep\lambda_{n(k)}}
e^{-\frac{n(k)}{2}\Rep\lambda_{(k)}}
}\\
\shoveleft{
\int\limits_{\left\{\lambda\in\sigma(A)\,\middle|\,e^{t\Rep\lambda}>n\right\}\cap \Delta_{n(k)}}1\,dv(e_{n(k)},g^*,\lambda)
}\\
\hfill
\text{since, by \eqref{infinity}, there is an $L>0$ such that
$e^{-\left[\frac{n(k)}{2}-t\right]\Rep\lambda_{n(k)}}\le L$, $k\in\N$;}
\\
\shoveleft{
\le Le^{-t}\sup_{\{g^*\in X^*\,|\,\|g^*\|=1\}}\sum_{k=1}^\infty e^{-\frac{n(k)}{2}\Rep\lambda_{n(k)}}
\int\limits_{\left\{\lambda\in\sigma(A)\,\middle|\,e^{t\Rep\lambda}>n\right\}\cap \Delta_{n(k)}}1\,dv(e_{n(k)},g^*,\lambda)
}\\
\hfill
\text{by \eqref{subvectors12};}
\\
\shoveleft{
= Le^{-t}\sup_{\{g^*\in X^*\,|\,\|g^*\|=1\}}\sum_{k=1}^\infty
\int\limits_{\left\{\lambda\in\sigma(A)\,\middle|\,e^{t\Rep\lambda}>n\right\}\cap \Delta_{n(k)}}1\,dv(E_A(\Delta_{n(k)})h,g^*,\lambda)
}\\
\hfill
\text{by \eqref{decompose};}
\\
\shoveleft{
= Le^{-t}\sup_{\{g^*\in X^*\,|\,\|g^*\|=1\}}
\int\limits_{\left\{\lambda\in\sigma(A)\,\middle|\,e^{t\Rep\lambda}>n\right\}}1\,dv(E_A(\cup_{k=1}^\infty\Delta_{n(k)})h,g^*,\lambda)
}\\
\hfill
\text{by \eqref{subvectors12};}
\\
\shoveleft{
=Le^{-t}\sup_{\{g^*\in X^*\,|\,\|g^*\|=1\}}\int\limits_{\{\lambda\in\sigma(A)\,|\,e^{t\Rep\lambda}>n\}}1\,dv(h,g^*,\lambda)
\hfill
\text{by \eqref{cond(ii)};}
}\\
\shoveleft{
\le Le^{-t}\sup_{\{g^*\in X^*\,|\,\|g^*\|=1\}}4M
\left\|E_A\left(\left\{\lambda\in\sigma(A)\,\middle|\,e^{t\Rep\lambda}>n\right\}\right)h\right\|\|g^*\|
}\\
\shoveleft{
\le 4LMe^{-t}\|E_A(\{\lambda\in\sigma(A)\,|\,e^{t\Rep\lambda}>n\})h\|
}\\
\hfill
\text{by the strong continuity of the {\it s.m.};}
\\
\hspace{1.2cm}
\to 4LMe^{-t}\left\|E_A\left(\emptyset\right)h\right\|=0,\ n\to\infty.
\hfill
\end{multline}

By Proposition \ref{prop}, \eqref{first2}, \eqref{ffirst2}, \eqref{second2}, and \eqref{ssecond2} jointly imply that 
\[
f\in \bigcap\limits_{t\in\R}D(e^{tA}),
\]
and hence, by Theorem \ref{GWS},
\[
y(t):=e^{tA}f,\ t\in\R,
\]
is a weak solution of equation \eqref{1}.

Since, for any $\lambda \in \Delta_{n(k)}$, $k\in \N$, by \eqref{radii1}, \eqref{infinity},
\begin{multline*}
\Rep\lambda =\Rep\lambda_{n(k)}-(\Rep\lambda_{n(k)}-\Rep\lambda)
\ge
\Rep\lambda_{n(k)}-|\Rep\lambda_{n(k)}-\Rep\lambda|
\\
\ \ \
\ge 
\Rep\lambda_{n(k)}-\varepsilon_{n(k)}
\ge \Rep\lambda_{n(k)}-1/n(k)\ge k-1\ge 0
\hfill
\end{multline*}
and, by \eqref{disks1},
\[
\Rep\lambda<n(k)^{-2}|\Imp\lambda|^{1/\beta},
\]
we infer that, for any $\lambda \in \Delta_{n(k)}$, $k\in \N$,
\begin{equation*}
|\lambda|\ge|\Imp\lambda|\ge 
\left[n(k)^2\Rep\lambda\right]^\beta\ge \left[n(k)^2(\Rep\lambda_{n(k)}-1/n(k))\right]^\beta.
\end{equation*}

Using this estimate, for an arbitrary $s>0$ and the functional $h^*\in X^*$ defined by \eqref{functional1}, we have:
\begin{multline}\label{notin}
\int\limits_{\sigma(A)}e^{s|\lambda|^{1/\beta}}\,dv(f,h^*,\lambda)
\hfill
\text{by \eqref{decompose} as in \eqref{first1};}
\\
\shoveleft{
=\sum_{k=1}^\infty e^{-n(k)\Rep\lambda_{n(k)}}\int\limits_{\sigma(A)\cap\Delta_{n(k)}}e^{s|\lambda|^{1/\beta}}\,dv(e_{n(k)},h^*,\lambda)
}\\
\shoveleft{
\ge\sum_{k=1}^\infty e^{-n(k)\Rep\lambda_{n(k)}}e^{sn(k)^2(\Rep\lambda_{n(k)}-1/n(k))}v(e_{n(k)},h^*,\Delta_{n(k)})
}\\
\shoveleft{
\ge \sum_{k=1}^\infty e^{-n(k)\Rep\lambda_{n(k)}}e^{sn(k)^2(\Rep\lambda_{n(k)}-1/n(k))}|\langle E_A(\Delta_{n(k)})e_{n(k)},h^*\rangle|
}\\
\hfill
\text{by \eqref{ortho1} and \eqref{funct-dist1};}
\\
\hspace{1.2cm}
\ge \sum_{k=1}^\infty \varepsilon
e^{(sn(k)-1)n(k)\Rep\lambda_{n(k)}-sn(k)}n(k)^{-2}
=\infty.
\hfill
\end{multline} 

Indeed, for all $k\in\N$ sufficiently large so that 
\begin{equation*}
sn(k)\ge s+2,
\end{equation*}
in view of \eqref{infinity},
\begin{multline*}
e^{(sn(k)-1)n(k)\Rep\lambda_{n(k)}-sn(k)}n(k)^{-2}
\ge 
e^{(s+1)n(k)-sn(k)}n(k)^{-2}=e^{n(k)}n(k)^{-2}
\\
\ \
\to\infty,\ k\to\infty.
\hfill
\end{multline*}

Whence, by Proposition \ref{prop} and \eqref{GC}, we infer that
\begin{equation*}
y(0)=f\notin \bigcup_{s>0} D(e^{s|A|^{1/\beta}})
={\mathscr E}^{\{\beta\}}(A),
\end{equation*}
which, by Proposition \ref{particular}, implies that the weak solution $y(t)=e^{tA}f$, $t\in\R$, 
of equation \eqref{1} does not belong to the Roumieu-type Gevrey class ${\mathscr E}^{\{\beta \}}\left( \R,X\right)$. 

The remaining case of
\[
\Rep\lambda_{n(k)}\to -\infty,\ k\to \infty
\] 
is symmetric to the case of
\[
\Rep\lambda_{n(k)}\to \infty,\ k\to \infty
\] 
and is considered in absolutely the same manner, which furnishes a weak solution $y(\cdot)$ of equation \eqref{1} such that
\begin{equation*}
y(0)=f\notin \bigcup_{s>0} D(e^{s|A|^{1/\beta}})
={\mathscr E}^{\{\beta\}}(A),
\end{equation*}
and hence, by Proposition \ref{particular}, not belonging to the Roumieu-type Gevrey class 
${\mathscr E}^{\{\beta \}}\left(\R,X\right)$.

With every possibility concerning $\{\Rep\lambda_n\}_{n=1}^\infty$ considered, 
the proof by contrapositive of the \textit{``only if" part} is complete and so is the proof of the 
entire statement.
\end{proof}

\begin{rem}
Due to the {\it scalar type spectrality} of the operator $A$, Theorem \ref{real} is stated exclusively in terms of the location of its {\it spectrum} in the complex plane, and hence, is an intrinsically qualitative statement (cf. \cite{Markin2018(3),Markin2018(4),Markin2019(2)}).
\end{rem}

From Theorem \ref{real} and {\cite[Theorem $3.1$]{Markin2019(2)}}, the latter characterizing the Roumieu-type strong Gevrey ultradifferentiability of all weak solution of equation \eqref{+} of order $\beta\ge 1$ on $(0,\infty)$, we derive

\begin{cor}\label{case+open}
Let $A$ be a scalar type spectral operator in a complex Banach space and $ 1\le \beta<\infty$. All weak solutions of equation \eqref{+} are $\beta$th-order Roumie-type Gevrey ultradifferentiable on $(0,\infty)$ iff all weak solutions of equation \eqref{1} are $\beta$th-order Beurling-type Gevrey ultradifferentiable on $\R$.
\end{cor}

For $\beta=1$, we obtain the following important particular case.

\begin{cor}[Characterization of the Entireness of Weak Solutions]\label{CEWS}\ \\
Let $A$ be a scalar type spectral operator in a complex Banach space. Every weak solution of equation \eqref{1} is an entire vector function iff there exist $b_+>0$ and $ b_->0$ such that the set $\sigma(A)\setminus {\mathscr P}^1_{b_-,b_+}$,
where
\begin{equation*}
{\mathscr P}^1_{b_-,b_+}:=\left\{\lambda \in \C\, \middle|\,
\Rep\lambda \le -b_-|\Imp\lambda| 
\ \text{or}\ 
\Rep\lambda \ge b_+|\Imp\lambda|\right\},
\end{equation*}
is bounded 
(see Figure \ref{fig:graph6}).

\begin{figure}[h]
\centering
\includegraphics[height=2in]{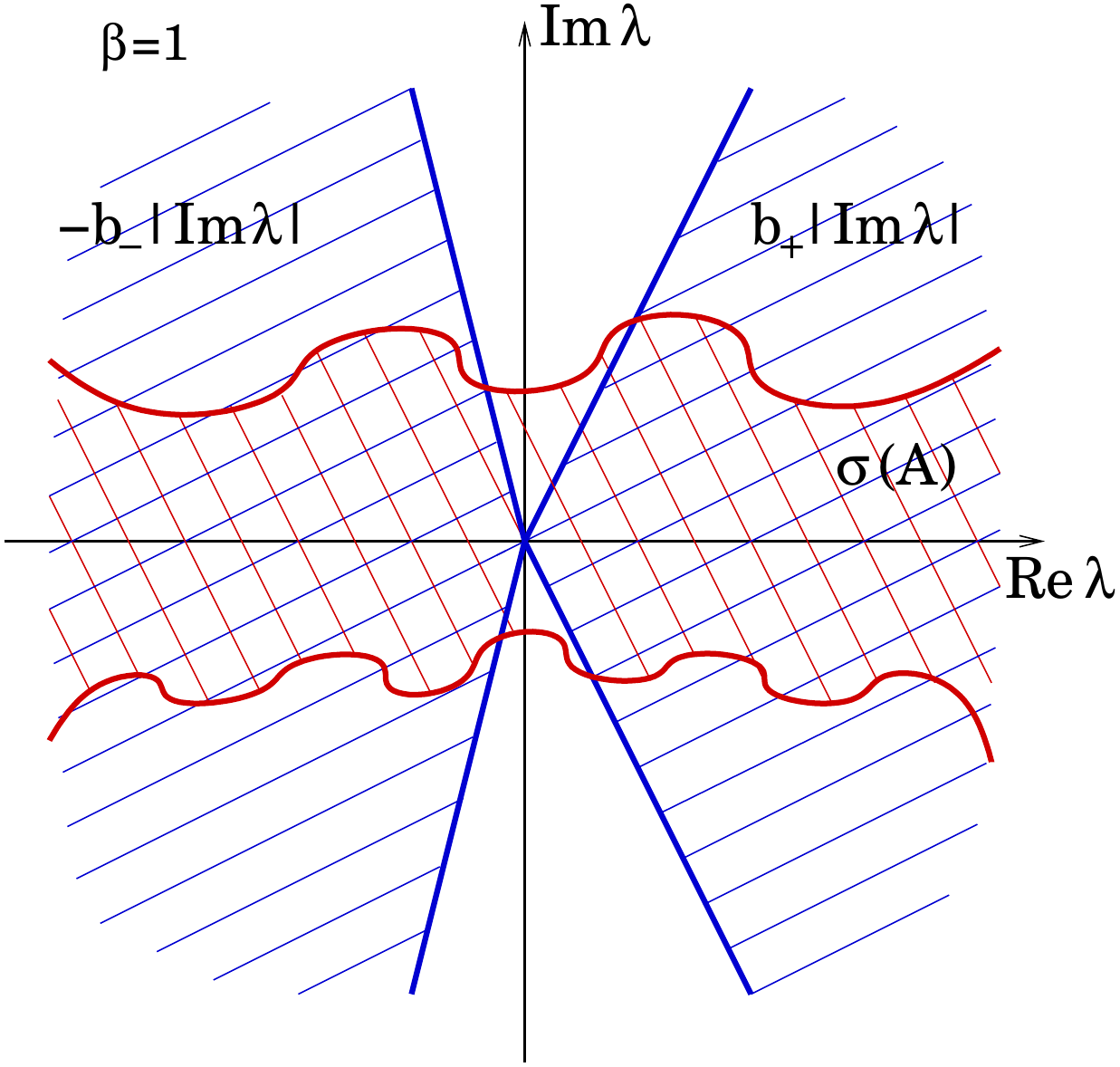}
\caption[]{The case of $\beta=1$.}
\label{fig:graph6}
\end{figure}
\end{cor}

\begin{rem}\label{remcews}
As is established in \cite{Markin2018(3)}, all weak solutions of equation \eqref{1} with a scalar type spectral operator $A$ in a complex Banach space can 
be \textit{strongly infinite differentiable}, while the operator $A$ is \textit{unbounded}. Corollary \ref{CEWS} further implies that all weak solutions of equation \eqref{1} with a scalar type spectral operator $A$ in a complex Banach space can be \textit{entire} vector functions without $A$ being bounded, e.g., when $A$ is an unbounded \textit{self-adjoint} operator in a complex Hilbert space (cf. Corollary \ref{real self-adjoint}, \cite[Corollary $4.1$]{Markin2001(1)}, and \cite[Corollary $5.1$]{Markin2001(1)}). This fact contrasts the situation when a closed densely defined linear operator $A$ in a complex Banach space generates a strongly continuous group $\left\{T(t) \right\}_{t\in \R}$ of bounded linear operators, i.e., the associated \textit{ACP} is \textit{well-posed} (see Remarks \ref{remsws}), in which case even the (left or right) strong differentiability of all weak solutions of equation \eqref{1} at $0$ immediately implies \textit{boundedness} for $A$ (cf., e.g., \cite{Engel-Nagel}).
\end{rem}

\section{Inherent Smoothness Improvement Effects}

Theorem \ref{real} implies, in particular, that
\begin{quote}
\textit{if, for some $ 1\le \beta<\infty$, every weak solution of equation \eqref{1} with a scalar type spectral operator $A$ in a complex Banach space $X$ belongs to the $\beta$th-order Roumieu-type Gevrey class 
${\mathscr E}^{\{\beta \}}\left(\R,X\right)$, then all of them belong to the narrower $\beta$th-order Beurling-type Gevrey class 
${\mathscr E}^{(\beta)}\left(\R,X\right)$}.
\end{quote}

Such a jump-like inherent smoothness improvement effect replicates the situation of the strong Gevrey ulradifferentiabilty of weak solutions of equation  \eqref{+} on $[0,\infty)$ {\cite[Theorem $4.1$]{Markin2018(4)}}.
 
Notably, for $\beta =1$, we have: 
\begin{quote}
\textit{if every weak solution of equation \eqref{1} with a scalar type spectral operator $A$ in a complex Banach space $X$ is analytically continuable into a complex neighborhood of $\R$ (each one into its own), then all of them are entire vector functions},
\end{quote}
which can be further strengthened as follows.

\begin{prop}[Smoothness Improvement Effect]\label{smimp1}\ \\
Let $A$ be a scalar type spectral operator in a complex Banach space $(X,\|\cdot\|)$. If every weak solution of equation \eqref{1} is analytically continuable into a complex neighborhood of $0$ (each one into its own), then all of them are entire vector functions.
\end{prop}

\begin{proof}
Let us first show that, if a weak solution $y(\cdot)$ of equation \eqref{1} is analytically continuable into a complex neighborhood of $0$, then $y(0)$ is an {\it analytic vector} of the operator $A$, i.e.,
\begin{equation*}
y(0)\in {\mathscr E}^{\{1\}}(A).
\end{equation*}

Let a weak solution $y(\cdot)$ of equation \eqref{1}
be analytically continuable
into a complex neighborhood of $0$. This implies that there is a $\delta>0$ such that
\begin{equation*}
y(t)=\sum_{n=0}^\infty \dfrac{y^{(n)}(0)}{n!}t^n,\ t\in [-\delta,\delta].
\end{equation*}

The power series converging at $t=\delta$, there is a $c>0$ such that
\begin{equation*}
\biggl\|\dfrac{y^{(n)}(0)}{n!}\delta^n\biggr\|\le c,\ n\in Z_+.
\end{equation*}

Whence, considering that, by Proposition \ref{particular} with $I=[-\delta,\delta]$,
\[
y(0)\in C^\infty(A)\ \text{and}\ y^{(n)}(0)=A^n y(0),\
n\in\Z_+,
\]
we infer that
\begin{equation*}
\|A^ny(0)\|=\|y^{(n)}(0)\|\le c\left[\delta^{-1}\right]^n n!,\ n\in\Z_+,
\end{equation*}
which implies
\begin{equation*}
y(0)\in {\mathscr E}^{\{1\}}(A).
\end{equation*}

Now, let us prove the statement {\it by contrapositive}
assuming that there is a weak solution of 
equation \eqref{1}, which is not an entire vector function. This, by Theorem \ref{real} with $\beta=1$, implies that there is a weak solution $y(\cdot)$ of 
equation \eqref{1}, which is not analytically continuable into a complex neighborhood of $\R$. Then, by Proposition \ref{particular}, for some
$t_0\in\R$,
\[
y(t_0)\not \in {\mathscr E}^{\{1\}}(A).
\]

Therefore, for the weak solution
\[
y_{t_0}(t):=y(t+t_0),\ t\in\R,
\]
of equation \eqref{1},
\[
y_{t_0}(0)=y(t_0)\notin {\mathscr E}^{\{1\}}(A),
\]
which, as is shown above, implies that $y_{t_0}(\cdot)$
is not analytically continuable into a complex neighborhood of $0$, and hence, completes the proof by contrapositive.
\end{proof}

Proposition \ref{smimp1} 
replicates the similar situation for weak solutions of equation  \eqref{+} {\cite[Proposition $5.1$]{Markin2018(4)}}.

\section{Gevrey Ultradifferentiability of Orders Less Than One}

While Corollary \ref{CEWS} implies that all weak solutions of equation \eqref{1} with a scalar type spectral operator $A$ in a complex Banach space can be \textit{entire} vector functions without the operator's $A$ being bounded (see Remark \ref{remcews}), the following analogue
of \cite[Theorem $4.1$]{Markin2019(1)} shows that $A$ cannot remain unbounded, if each weak solution $y(\cdot)$ of \eqref{1}, in addition to being entire, satisfies the growth estimate 
\begin{equation*}
\|y(z)\|\le Me^{\gamma|z|^{1/(1-\beta)}},\ z\in \C,
\end{equation*}
with some $0\le \beta<1$, $\gamma>0$, and $M>0$
depending on $y(\cdot)$ (see \eqref{order}).

\begin{thm}[Gevrey Ultradifferentiability of Orders Less Than One]\label{ol1}\ \\
Let $A$ be a scalar type spectral operator in a complex Banach space $(X,\|\cdot\|)$ with spectral measure $E_A(\cdot)$. If every weak solution of equation \eqref{1} belongs to the $\beta$th-order Roumieu-type Gevrey class ${\mathscr E}^{\{\beta \}}\left(\R,X\right)$ with some $\beta\in [0,1)$ (each one to its own), then the operator $A$ is bounded, and hence, all weak solutions of equation \eqref{1} are necessarily entire vector functions of exponential type.
\end{thm}

\begin{proof}
For the scalar type spectral operators
\[
A_-:=AE_A\left(\left\{\lambda \in \sigma(A)\,
\middle|\,\Rep\lambda< 0\right\}\right)
\]
and
\[
A_+:=AE_A\left(\left\{\lambda \in \sigma(A)\,
\middle|\,\Rep\lambda\ge 0\right\}\right),
\]
by the properties of the operational calculus (see
\cite[Theorem XVIII.$2.11$]{Dun-SchIII}),
\[
A=A_-+A_+.
\]

Let
\begin{equation*}
A_1:=\restr{A}{E_A\left(\left\{\lambda \in \sigma(A)\,
\middle|\,\Rep\lambda< 0\right\}\right)X}
\end{equation*}
and 
\begin{equation*}
A_2:=\restr{A}{E_A\left(\left\{\lambda \in \sigma(A)\,
\middle|\,\Rep\lambda\ge 0\right\}\right)X}
\end{equation*}
where $\restr{\cdot}{\cdot}$ is the \textit{restriction} of an operator (left) to a subspace (right).

Since, by the properties of the operational calculus, the complementary subspaces
\[
E_A\left(\left\{\lambda \in \sigma(A)\,
\middle|\,\Rep\lambda\ge 0\right\}\right)X
\quad \text{and}\quad
E_A\left(\left\{\lambda \in \sigma(A)\,
\middle|\,\Rep\lambda\ge 0\right\}\right)X
\]
reduce the operators $A$, $A_-$, and $A_+$,
\[
\begin{split}
&\sigma(A)=\sigma(A_1)\cup \sigma(A_2),\\
&\sigma(A_1)\subseteq \sigma(A_-)\subseteq \sigma(A_1)\cup \left\{ 0\right\},\\
&\sigma(A_2)\subseteq \sigma(A_+)\subseteq \sigma(A_2)\cup \left\{ 0\right\}
\end{split}
\]
(cf. \cite{Markin2017}), with $\sigma(A_i)=\emptyset$, $i=1,2$, provided the underlying subspace is $\left\{ 0\right\}$.

Therefore, we have the following inclusions:
\begin{equation}\label{incl}
\sigma(A)\subseteq \sigma(A_-)\cup \sigma(A_+)\subseteq \sigma(A)\cup \{0\}.
\end{equation}

By the properties of the operational calculus (see
\cite[Theorem XVIII.$2.11$]{Dun-SchIII}),
\begin{multline}\label{ol11}
e^{tA_+}=
\int\limits_{\sigma(A)} e^{t\lambda \chi_{\left\{\lambda \in \sigma(A)\,
\middle|\,\Rep \lambda\ge 0\right\}}(\lambda)}\,dE_A(\lambda)
\\
\shoveleft{
=
\int\limits_{\left\{\lambda \in \sigma(A)\,
\middle|\,\Rep\lambda\ge 0\right\}} e^{t\lambda}\,dE_A(\lambda)
+
\int\limits_{\left\{\lambda \in \sigma(A)\,
\middle|\,\Rep \lambda<0\right\}} 1\,dE_A(\lambda)
}\\
\ \qquad \quad
=e^{tA}E_A\left(\left\{\lambda \in \sigma(A)\,
\middle|\,\Rep\lambda\ge 0\right\}\right)
+E_A\left(\left\{\lambda \in \sigma(A)\,
\middle|\,\Rep\lambda<0\right\}\right),\  t\in \R,
\hfill
\end{multline}
($\chi_\delta(\cdot)$ is the {\it characteristic function} of a set $\delta\subseteq \C$).

Let
\[
f\in \bigcap_{t\ge 0}D\left(e^{tA_+}\right)
\]
be arbitrary. Then, by \eqref{ol11},
\[
E_A\left(\left\{\lambda \in \sigma(A)\,
\middle|\,\Rep\lambda\ge 0\right\}\right)f\in 
\bigcap_{t\ge 0}D\left(e^{tA}\right),
\]

Since, for arbitrary $t\in \R$ and any Borel set $\delta\subseteq \C$,
\begin{equation*}
\int\limits_\delta e^{t\lambda}\,dv(E_A\left(\left\{\lambda \in \sigma(A)\,
\middle|\,\Rep\lambda\ge 0\right\}\right)f,g^*,\lambda)
=\int\limits_{\delta\cap\left\{\lambda \in \sigma(A)\,
\middle|\,\Rep\lambda\ge 0\right\}}e^{t\lambda}\,dv(f,g^*,\lambda)
\end{equation*}
(see, e.g., \cite[Preliminaries]{Markin2019(1)}), by Proposition \ref{prop}, we infer that
\[
E_A\left(\left\{\lambda \in \sigma(A)\,
\middle|\,\Rep\lambda\ge 0\right\}\right)f\in 
\bigcap_{t<0}D(e^{tA}).
\]

Hence, for any $f\in \bigcap_{t\ge 0}D\left(e^{tA_+}\right)$,
\[
E_A\left(\left\{\lambda \in \sigma(A)\,
\middle|\,\Rep\lambda\ge 0\right\}\right)f\in 
\bigcap_{t\in \R}D(e^{tA}).
\]

This, by \cite[Theorem $4.2$]{Markin2002(1)}, implies that
every weak solution 
\begin{equation*}
e^{tA_+}f=e^{tA}E_A\left(\left\{\lambda \in \sigma(A)\,
\middle|\,\Rep\lambda\ge 0\right\}\right)f
+E_A\left(\left\{\lambda \in \sigma(A)\,
\middle|\,\Rep\lambda<0\right\}\right)f,\  t\ge 0,
\end{equation*}
where
\[
f\in\bigcap_{t\ge 0}D\left(e^{tA_+}\right)
\]
is arbitrary, of the equation 
\begin{equation*}
y'(t)=A_+y(t),\ t\ge 0,
\end{equation*}
along with the weak solution
\[
e^{tA}E_A\left(\left\{\lambda \in \sigma(A)\,
\middle|\,\Rep\lambda\ge 0\right\}\right)f,\ t\in \R,
\]
of equation \eqref{1} and the constant vector function
\[
E_A\left(\left\{\lambda \in \sigma(A)\,
\middle|\,\Rep\lambda<0\right\}\right)f,\ t\in \R,
\]
belongs to the $\beta$th-order Roumieu-type Gevrey class ${\mathscr E}^{\{\beta \}}\left([0,\infty),X\right)$ with some $\beta\in [0,1)$ (each one to its own), which, by \cite[Theorem $4.1$]{Markin2019(1)}, implies that the operator $A_+$ is \textit{bounded}, and hence (see, e.g., \cite{Dun-SchI,Markin2020EOT}), the spectrum
$\sigma(A_+)$ is a \textit{bounded set} in $\C$. 

Using similar reasoning for the scalar type spectral operator $-A_-$ and the evolution equation
\[
y'(t)=-A_-y(t),\ t\ge 0,
\]
(see Remarks \ref{remsws}), one can show that the spectrum of the operator $-A_-$, and hence, of the operator $A_-$ is a \textit{bounded set} in $\C$. 

Thus, from inclusion \eqref{incl}, we infer that $\sigma(A)$ is a \textit{bounded set} in $\C$, which, by \cite[Theorem XVIII.$2.11$ (c)]{Dun-SchIII}, means that the operator $A$ is \textit{bounded} and completes the proof
implying that every weak solution $y(\cdot)$ of equation \eqref{1} is an entire vector function of the form
\[
y(z)=e^{zA}f=\sum_{n=0}^\infty \dfrac{z^n}{n!}A^nf,\ z\in \C,\ \text{with some}\ f\in X,
\] 
and hence, satisfying the growth condition
\[
\|y(z)\|\le \|f\|e^{\|A\||z|}, \ z\in \C,
\]
is of \textit{exponential type} (see Preliminaries).
\end{proof}

\section{The Case of a Normal Operator}

As an important particular case of Theorem \ref{real}, we  obtain

\begin{cor}[Gevrey Ultradifferentiability of order $\beta \ge 1$]\label{realnormal}\ \\
Let $A$ be a normal operator in a complex Hilbert space $X$ and $ 1\le \beta<\infty$. Then the following statements are equivalent. 
\begin{enumerate}[label={(\roman*)}]
\item Every weak solution of equation \eqref{1} belongs to the $\beta$th-order Beurling-type Gevrey class 
${\mathscr E}^{(\beta )}\left(\R,X\right)$.
\item Every weak solution of equation \eqref{1} belongs to the $\beta$th-order Roumieu-type Gevrey class 
${\mathscr E}^{\{\beta \}}\left(\R,X\right)$.
\item There exist
$b_+>0$ and $ b_->0$ such that the set $\sigma(A)\setminus {\mathscr P}^\beta_{b_-,b_+}$,
where
\begin{equation*}
{\mathscr P}^\beta_{b_-,b_+}:=\left\{\lambda \in \C\, \middle|\,
\Rep\lambda \le -b_-|\Imp\lambda|^{1/\beta} 
\ \text{or}\ 
\Rep\lambda \ge b_+|\Imp\lambda|^{1/\beta} \right\},
\end{equation*}
is bounded (see Figure \ref{fig:graph4}).
\end{enumerate}
\end{cor}

\begin{rem}
Corollary \ref{realnormal} develops the results of paper \cite{Markin2001(1)}, where similar consideration is given to the strong Gevrey ultradifferentiability of the weak solutions of equation \eqref{+} with a normal operator in a complex Hilbert space on $[0,\infty)$ and $(0,\infty)$.
\end{rem}

For $\beta=1$, we obtain the following important particular case.

\begin{cor}[Characterization of the Entireness of Weak Solutions]\label{CEWSN}\ \\
Let $A$ be a normal operator in a complex Hilbert space. Every weak solution of equation \eqref{1} is an entire vector function iff there exist $b_+>0$ and $ b_->0$ such that the set $\sigma(A)\setminus {\mathscr P}^1_{b_-,b_+}$,
where
\begin{equation*}
{\mathscr P}^1_{b_-,b_+}:=\left\{\lambda \in \C\, \middle|\,
\Rep\lambda \le -b_-|\Imp\lambda| 
\ \text{or}\ 
\Rep\lambda \ge b_+|\Imp\lambda|\right\},
\end{equation*}
is bounded 
(see Figure \ref{fig:graph6}).
\end{cor}

Considering that, for a self-adjoint operator $A$ in a complex Hilbert space $X$,
\[
\sigma(A)\subseteq \R
\]
(see, e.g., \cite{Dun-SchII,Plesner}), by Corollary \ref{CEWSN}, we can strengthen {\cite[Corollary $18$]{Markin2018(3)}} as follows.

\begin{cor}[The Case of a Self-Adjoint Operator]\label{real self-adjoint}\ \\
Every weak solution of equation \eqref{1} 
with a self-adjoint operator $A$ in a complex Hilbert space is an entire vector function.
\end{cor}

Cf. {\cite[Corollary $4.1$]{Markin2001(1)}} and {\cite[Corollary $5.1$]{Markin2001(1)}}.

From Corollary \ref{case+open}, we immediately obtain

\begin{cor}\label{casenormal+open}
Let $A$ be a normal operator in a complex Hilbert space and $ 1\le \beta<\infty$. All weak solutions of equation \eqref{+} are $\beta$th-order Roumie-type Gevrey ultradifferentiable on $(0,\infty)$ iff all weak solutions of equation \eqref{1} are $\beta$th-order Beurling-type Gevrey ultradifferentiable on $\R$.
\end{cor}

Cf. {\cite[Theorem $4.2$]{Markin2001(1)}}.

For a normal operator in a complex Hilbert space, Proposition \ref{smimp1} and Theorem \ref{ol1} acquire the following form, respectively.

\begin{cor}[Smoothness Improvement Effect]\label{smimp2}\ \\
Let $A$ be a normal operator in a complex Hilbert space. If every weak solution of equation \eqref{1} is analytically continuable into a complex neighborhood of $0$ (each one into its own), then all of them are entire vector functions.
\end{cor}   

\begin{cor}[Gevrey Ultradifferentiability of Orders Less Than One]\label{ol1no}\ \\
Let $A$ be a normal operator in a complex Hilbert space $X$. If every weak solution of equation \eqref{1} belongs to the $\beta$th-order Roumieu type Gevrey class ${\mathscr E}^{\{\beta \}}\left(\R,X\right)$ with $0\le \beta<1$ (each one to its own), then the operator $A$ is bounded, and hence, all weak solutions of equation \eqref{1} are necessarily entire vector functions of exponential type.
\end{cor}

Observe that Corollary \ref{smimp2} 
replicates the similar situation for weak solutions of equation  \eqref{+} with a normal operator in a complex Hilbert space {\cite[Proposition $5.1$]{Markin2001(1)}}.

\section{Acknowledgments}

The author's appreciation goes to his colleague, Dr.~Maria Nogin of the Department of Mathematics, California State University, Fresno, for her kind assistance with the graphics.


\end{document}